\documentclass[a4paper,12pt]{article}

\usepackage[mathscr]{eucal}

\usepackage{amscd}
\usepackage{amssymb,amsfonts,amsmath,amsthm}
\usepackage{verbatim}

\usepackage{amsmath}
\usepackage{amsthm}
\usepackage{amssymb}
\usepackage{graphicx}

\usepackage{latexsym}
\usepackage{array}
\usepackage{enumerate}

\numberwithin{equation}{section}

\newcommand{\shom}{{\mathcal{H}}om}

\newcommand{\CC}{\mathbb{C}}

\newcommand{\RR}{\mathbb{R}}
\newcommand{\QQ}{\mathbb{Q}}
\newcommand{\ZZ}{\mathbb{Z}}

\newcommand{\F}{\mathcal{F}}
\newcommand{\G}{\mathcal{G}}
\newcommand{\M}{\mathcal{M}}

\newcommand{\sho}{\mathcal{O}}
\newcommand{\R}{\mathcal{R}}

\newcommand{\PP}{{\mathbb P}}

\newcommand{\KK}{{\rm K}}

\newcommand{\HSm}{{\rm HS}^{\rm mon}}
\newcommand{\LL}{{\mathbb L}}
\newcommand{\height}{{\rm ht}}
\newcommand{\Var}{{\rm Var}}
\renewcommand{\(}{\left(}
\renewcommand{\)}{\right)}

\newcommand{\LK}{\mathrm{lk}}
\newcommand{\MCS}{\mathcal{S}}
\newcommand{\UH}{\mathrm{UH}}

\newcommand{\mer}{{\rm mero}}
\newcommand{\merc}{{\rm mero,c}}

\renewcommand{\dim}{{\rm dim}}

\newcommand{\Vol}{{\rm Vol}}

\newcommand{\e}{\varepsilon}
\newcommand{\Spec}{{\rm Spec}}

\newcommand{\id}{{\rm id}}
\newcommand{\supp}{{\rm supp}}

\newcommand{\Int}{{\rm Int}}
\newcommand{\relint}{{\rm rel.int}}

\newcommand{\Db}{{\bf D}^{\mathrm{b}}}
\newcommand{\Dbc}{{\bf D}_{\mathrm{c}}^{\mathrm{b}}}

\newcommand{\tl}[1]{\widetilde{#1}}

\newcommand{\simto}{\overset{\sim}{\longrightarrow}}

\newcommand{\dsum}{\displaystyle \sum}

\newcommand{\GR}{\mathrm{Gr}}

\newtheorem{theorem}{Theorem}[section]

\newtheorem{corollary}[theorem]{Corollary}
\newtheorem{lemma}[theorem]{Lemma}
\newtheorem{proposition}[theorem]{Proposition}

\theoremstyle{definition}
\newtheorem{definition}[theorem]{Definition}
\theoremstyle{remark}
\newtheorem{remark}[theorem]{\sc Remark}
\newtheorem{example}[theorem]{\sc Example}
\title{Meromorphic nearby cycle functors and 
monodromies of meromorphic functions \\ 
(with Appendix by T. Saito) 
}

\author{Tat Thang NGUYEN 
\footnote{Institute of Mathematics, 
Vietnam Academy of Science and Technology, 18 
Hoang Quoc Viet Road, Cau Giay District, Hanoi, Vietnam. 
E-mail: ntthang@math.ac.vn}, 
and Kiyoshi TAKEUCHI 
\footnote{Mathematical Institute, Tohoku University, 
Aramaki Aza-Aoba 6-3, Aobaku, Sendai, 980-8578, Japan. 
E-mail: takemicro@nifty.com} 
}

\date{}

\begin{document}

\maketitle

\begin{abstract}
We introduce meromorphic nearby cycle functors and 
study their functorial properties. Moreover we apply them 
to monodromies of meromorphic functions in various 
situations. Combinatorial descriptions of their 
reduced Hodge spectra and 
Jordan normal forms will be obtained. 
\end{abstract}

\maketitle

\section{Introduction}\label{sec:s1}

\par
In \cite{G-L-M} Gusein-Zade, Luengo and Melle-Hern\'andez 
generalized Milnor's fibration theorem to meromorphic functions 
and defined their Milnor fibers. Moreover they obtained 
a formula for their monodromy zeta functions. 
Since then many authors studied Milnor fibers of 
meromorphic functions (see e.g. \cite{B-P}, 
\cite{B-P-S}, \cite{GV-L-M}, \cite{L-W}, \cite{Thang}, \cite{P}, 
\cite{Raibaut-2} \cite{S-T-1}, 
\cite{Tibar} etc.). However, in contrast to 
Milnor fibers of holomorphic functions, 
the geometric structures of those 
of meromorphic functions look much more complicated. 
For example, Milnor \cite{Milnor} proved 
that if a holomorphic function 
has an isolated singular point then the Milnor 
fiber at it has the homotopy type of a bouquet of 
some spheres. This implies that its reduced 
cohomology groups are concentrated in the 
middle dimension. To the best of our knowledge, 
we do not know so far such a nice structure theorem for 
Milnor fibers of meromorphic functions. 
For this reason, we cannot know any property 
of each Milnor monodromy operator of 
a meromorphic function even if we 
have a formula for its monodromy zeta function. 
We have also a serious obstruction that the theory of 
nearby and vanishing cycle functors for 
meromorphic functions is not fully 
developed yet. Indeed, nowadays 
the corresponding functors 
for holomorphic functions are 
not only indispensable 
for the study of Milnor monodromies 
but also very useful in many fields of 
mathematics. 

 In this paper, we overcome the above-mentioned problems 
partially by laying a foundation of the 
theory of nearby cycle functors for 
meromorphic functions. In particular, we prove 
that they preserve the perversity as 
in the holomorphic case. Then we apply 
our new algebraic machineries to Milnor monodromies of 
meromorphic functions. In this way, we 
obtain various new results on them especially 
for the eigenvalues $\lambda \not= 1$. 
In order to describe our results more precisely, 
from now we prepare some notations. 
Let $X$ be a complex manifold and $P(x), Q(x)$ 
holomorphic functions on it. Assume that $Q(x)$ 
is not identically zero on each connected component 
of $X$. Then we define a meromorphic function $f(x)$ on 
$X$ by 
\begin{equation}
f(x)= \frac{P(x)}{Q(x)}  \qquad (x \in X). 
\end{equation}
Let us set $I(f)=P^{-1}(0) \cap Q^{-1}(0) \subset X$. 
If $P$ and $Q$ are coprime in the local ring 
$\sho_{X,x}$ at a point $x \in X$, then 
$I(f)$ is nothing but the set of the indeterminacy 
points of $f$ on a neighborhood of $x$. Note that 
the set $I(f)$ depends on the pair $(P(x), Q(x))$ of 
holomorphic functions representing $f(x)$. 
For example, if we take a holomorphic function 
$R(x)$ on $X$ (which is not identically zero on 
each connected component of $X$) and set 
\begin{equation}
g(x)= \frac{P(x)R(x)}{Q(x)R(x)}  \qquad (x \in X), 
\end{equation}
then the set $I(g)=I(f) \cup R^{-1}(0)$ might be 
bigger than $I(f)$. In this way, we 
distinguish $f(x)= \frac{P(x)}{Q(x)}$ from 
$g(x)= \frac{P(x)R(x)}{Q(x)R(x)}$ even if 
their values coincide over an open dense 
subset of $X$. 
This is the convention due to 
Gusein-Zade, Luengo and Melle-Hern\'andez 
\cite{G-L-M} etc. 
Now we recall the following 
fundamental theorem due to \cite{G-L-M}. 

\begin{theorem}\label{the-fib} 
(Gusein-Zade, Luengo and Melle-Hern\'andez 
\cite{G-L-M}) 
For any point $x \in P^{-1}(0)$ there exists 
$\e_0> 0$ such that for any $0< \e < \e_0$ 
and the open ball $B(x; \e ) \subset X$ of 
radius $\e >0$ with center at $x$ 
(in a local chart of $X$) the restriction 
\begin{equation}
B(x; \e ) \setminus Q^{-1}(0) 
\longrightarrow \CC
\end{equation}
of $f: X \setminus Q^{-1}(0) 
\longrightarrow \CC$ is a locally trivial 
fibration over a sufficiently small 
punctured disk in $\CC$ with center at 
the origin $0 \in \CC$ 
\end{theorem}

We call the fiber in this theorem the Milnor fiber of 
the meromorphic function $f(x)= \frac{P(x)}{Q(x)}$ 
at $x \in P^{-1}(0)$ and denote it by $F_x$. 
As in the holomorphic case, we obtain also 
its Milnor monodromy operators
\begin{equation}
\Phi_{j,x}: H^j(F_x; \CC ) \simto H^j(F_x; \CC ) \qquad 
(j \in \ZZ ). 
\end{equation}
Then we define 
the monodromy zeta function $\zeta_{f,x}(t) \in \CC (t)$ 
of $f$ at $x \in P^{-1}(0)$ by 
\begin{equation}
\zeta_{f,x}(t) = \prod_{j \in \ZZ} 
\Bigl\{ {\rm det}( \id - t \Phi_{j,x}) \Bigr\}^{(-1)^j}  \in \CC (t). 
\end{equation}
In \cite{G-L-M} Gusein-Zade, Luengo and Melle-Hern\'andez 
obtained a formula which expresses $\zeta_{f,x}(t) \in \CC (t)$ 
in terms of the Newton polyhedra of $P$ and $Q$ at $x$ 
(for the details, see Theorem \ref{the-2} below). 
However it is not possible to deduce any property of 
each monodromy operator $\Phi_{j,x}$ from it. 
From now, we shall explain how we can 
overcome this problem for the 
eigenvalues $\lambda \not= 1$ of $\Phi_{j,x}$. 
First we extend the classical 
notion of nearby cycle functors to 
meromorphic functions as follows 
(see also Raibaut \cite{Raibaut-2} for a 
similar but sligthly different approach to them). 
Denote by $\Db(X)$ the 
derived category whose objects are 
bounded complexes of sheaves 
of $\CC_X$-modules on $X$. 
For the meromorphic function $f(x)= \frac{P(x)}{Q(x)}$ let 
\begin{equation}
i_f: X \setminus Q^{-1}(0) \hookrightarrow 
X \times \CC_t
\end{equation}
be the (not necessarily) 
closed embedding defined by $x \mapsto (x,f(x))$. 
Let $t: X \times \CC \rightarrow \CC$ be the 
second projection. Then 
for $\F \in \Db(X)$ we set 
\begin{equation}
\psi_f^{\mer}( \F ):= \psi_t( Ri_{f*}
( \F |_{X \setminus Q^{-1}(0)}) ) 
\in \Db(X). 
\end{equation}
We call $\psi_f^{\mer}( \F )$ the meromorphic 
nearby cycle sheaf of $\F$ along $f$. 
Then as in the holomorphic case, 
for any point $x \in P^{-1}(0)$ and 
$j \in \ZZ$ we have an isomorphism 
(see Lemma \ref{lem-1}) 
\begin{equation}
H^j \psi_f^{\mer}( \F )_x \simeq 
H^j (F_x; \F ). 
\end{equation}
Moreover we will show 
(see Theorem \ref{the-1}) that the functor 
\begin{equation}
\psi_f^{\mer}( \cdot ) : 
\Db(X) \longrightarrow \Db(X)
\end{equation}
preserves the constructibility and 
the perversity (up to some shift). Then thanks to 
the perversity, we obtain the following theorem. 
The problem being local, we may 
assume that $X= \CC^n$ and $0 \in I(f)=P^{-1}(0) 
\cap Q^{-1}(0)$. For $j \in \ZZ$ and $\lambda \in \CC$ 
we denote by 
\begin{equation}
H^j(F_0; \CC )_{\lambda} \subset H^j(F_0; \CC )
\end{equation}
the generalized eigenspace of $\Phi_{j,0}$ 
for the eigenvalue $\lambda$. 

\begin{theorem}\label{the-3} 
Assume that the hypersurfaces $P^{-1}(0)$ and 
$Q^{-1}(0)$ of $X= \CC^n$ have an isolated 
singular point at the origin $0 \in X= \CC^n$ 
and intersect transversally on $X \setminus \{ 0 \}$. 
Then for any $\lambda \not=1$ we have 
the concentration 
\begin{equation}
H^j(F_0; \CC )_{\lambda} \simeq 0 
\qquad (j \not= n-1). 
\end{equation}
\end{theorem}

Combining the formula for 
$\zeta_{f,0}(t) \in \CC (t)$ in \cite{G-L-M} 
(see Theorem \ref{the-2} below) 
with our Theorem \ref{the-3}, we obtain a 
formula for the multiplicities of the eigenvalues 
$\lambda \not= 1$ in $\Phi_{n-1,0}$. It 
seems that there is some geometric background on 
$F_0$ (like Milnor's celebrated bouquet 
decomposition theorem in \cite{Milnor}) 
for Theorem \ref{the-3} to hold. 
It would be an interesting problem to know it 
and reprove Theorem \ref{the-3} in a purely 
geometric manner. From now on, 
we assume also that $f(x)= \frac{P(x)}{Q(x)}$ is 
a rational function. 
To obtain also a formula for 
the Jordan normal form of its monodromy 
$\Phi_{n-1, 0}$ as in 
Matsui-Takeuchi \cite{M-T-4}, Stapledon \cite{Stapledon} and 
Saito \cite{Saito}, we have to assume moreover that 
$f$ is polynomial-like in the following sense. 

\begin{definition}\label{def-2 7}
We say that the rational function 
$f(x)= \frac{P(x)}{Q(x)}$ is polynomial-like 
if there exists a 
resolution $\pi_0: \widetilde{X} \rightarrow X= \CC^n$
of singularities of $P^{-1}(0) \cup Q^{-1}(0)$ 
which induces an isomorphism 
$\widetilde{X} \setminus \pi_0^{-1}( \{ 0 \} ) 
\simto X \setminus \{ 0 \}$ such that for 
any irreducible component $D_i$ of the 
(exceptional) normal crossing divisor 
$D= \pi_0^{-1}( \{ 0 \} )$ we have the condition 
\begin{equation}
{\rm ord}_{D_i} (P \circ \pi_0) > 
{\rm ord}_{D_i} (Q \circ \pi_0). 
\end{equation}
\end{definition}

For typical examples of polynomial-like 
$f(x)= \frac{P(x)}{Q(x)}$, see Definition \ref{def-29}. 
If $f(x)$ is non-degenerate and 
satisfies the condition in it, then 
by a toric modification $\pi_0: \widetilde{X} 
\rightarrow X= \CC^n$ of $X= \CC^n$ we can 
check that it is polynomial-like in the above sense. 
Note that in the study of fibrations of mixed 
functions (of type $f \overline{g}$) 
recently Oka \cite{Oka-2} introduced a similar 
condition and called it the multiplicity condition. 
Then we obtain the following result. 

\begin{theorem}\label{weifilonofi}
In the situation of Theorem \ref{the-3}, 
assume also that $f$ is polynomial-like. Then 
for any $\lambda\neq 1$ 
the weight filtration of the mixed 
Hodge structure of $H^{n-1}(F_{0};
\CC)_{\lambda}$ is the monodromy weight 
filtration of the Milnor monodromy 
$\Phi_{n-1, 0}\colon H^{n-1}(F_{0};
\CC)_{\lambda}\overset{\sim}{\to} H^{n-1}(
F_0 ;\CC)_{\lambda}$ centered at $n-1$.
\end{theorem}

By this theorem, forgetting the eigenvalue $1$ 
parts of the Milnor monodromies $\Phi_{j, 0}$ 
we can define (see Definition \ref{def-HSP}) 
the reduced Hodge spectrum 
$\tl{\rm sp}_{f,0}(t)$ of 
the Milnor fiber $F_0$ which satisfies the symmetry 
\begin{equation}
\tl{\rm sp}_{f,0}(t) = t^n \cdot
\tl{\rm sp}_{f,0} \Bigl( \frac{1}{t} \Bigr)
\end{equation}
centered at $\frac{n}{2}$. 
Moreover, we define the motivic zeta function 
of the rational function $f$ and the 
motivic Milnor fiber as its limit. 
In fact, in \cite{Raibaut-2} Raibaut has 
defined the same objects earlier, but 
our proofs are different from his ones. 
In the last half of Section \ref{sect-2}, 
we also give more explicit formulas 
in terms of Newton polyhedrons. 
See Section \ref{sect-2} for the details. 
Then, assuming also that $f$ is non-degenerate 
at the origin $0 \in X= \CC^n$ and $P(x), Q(x)$ 
are convenient, we obtain 
combinatorial descriptions of the reduced Hodge spectrum 
$\tl{\rm sp}_{f,0}(t)$ of $F_0$ and the Jordan 
normal forms of $\Phi_{n-1, 0}$ for 
the eigenvalues $\lambda \not= 1$ as in 
Esterov-Takeuchi \cite{E-T}, Matsui-Takeuchi \cite{M-T-4}, 
Stapledon \cite{Stapledon} and 
Saito \cite{Saito}. See Section \ref{sec:7} 
for the details. We can also globalize 
these results and obtain 
similar formulas for monodromies at infinity of 
the rational function 
$f(x)= \frac{P(x)}{Q(x)}$. They are natural 
generalizations of the results in 
Libgober-Sperber \cite{L-S}, Matsui-Takeuchi 
\cite{M-T-2}, \cite{M-T-3}, Stapledon \cite{Stapledon} 
and Takeuchi-Tib{\u a}r \cite{T-T}. 
See Section \ref{sec:9} for the details.

\bigskip
\noindent{\bf Acknowledgement:} 
The authors would like to 
express their hearty gratitude to Professors 
David Massey and J\"org Sch\"urmann for useful discussions 
on the content of Section \ref{sec:s2}. 
Especially, by an idea of Professor 
David Massey we could simplify the proof of 
Theorem \ref{the-1}. The authors thank also Takahiro Saito 
for several fruitful discussions during the 
preparation of this paper. Especially, Lemma~\ref{weifillem} 
is due to him. Last but not least, they are very 
grateful to the anonymous referees whose suggestions 
have substentially improved this paper.  The first author 
is partially supported by Vietnam National Foundation for Science and 
Technology Development (NAFOSTED) under grant number 101.04-2019.305 
and by the bilateral joint research
project between Vietnam Academy of Science and Technology (VAST) 
and the Japan Society for the Promotion of Science (JSPS) 
under the Grant QTJP01.02/21-23.

\section{Meromorphic nearby cycle functors}\label{sec:s2}

In this section, we introduce 
meromorphic nearby cycle functors and 
study their functorial properties. 
In this paper we essentially 
follow the terminology of 
\cite{Dimca}, \cite{H-T-T} and \cite{K-S}. 
Let $k$ be an arbitrary field and for a topological 
space $X$ denote by $\Db(X)$ the 
derived category whose objects are 
bounded complexes of sheaves 
of $k_X$-modules on $X$. If $X$ is a complex manifold, 
we denote by $\Dbc(X)$ the full 
subcategory of $\Db(X)$ consisting of 
constructible objects. 
Let $X$, $P(x), Q(x)$ and $f(x)= \frac{P(x)}{Q(x)}$ etc. 
be as in Section \ref{sec:s1} and for $\F \in \Db(X)$ 
define the meromorphic 
nearby cycle sheaf $\psi_f^{\mer}( \F ) 
\in \Db(X)$ of $\F$ along $f$ as before. 
Then we have the following results. 

\begin{lemma}\label{lem-1} 
\begin{enumerate}
\item[\rm{(i)}] The support of $\psi_f^{\mer}( \F )$ is 
contained in $P^{-1}(0)$. 
\item[\rm{(ii)}] There exists an isomorphism 
\begin{equation}
\psi_f^{\mer}( \F ) \simto 
\psi_f^{\mer}( R 
\Gamma_{X \setminus (P^{-1}(0) \cup Q^{-1}(0))} 
( \F )). 
\end{equation}
\item[\rm{(iii)}] For any point $x \in P^{-1}(0)$ and 
$j \in \ZZ$ we have an isomorphism 
\begin{equation}
H^j \psi_f^{\mer}( \F )_x \simeq 
H^j (F_x; \F )
\end{equation}
compatible with the monodromy automorphisms 
on the both sides. 
\end{enumerate}
\end{lemma}

\begin{proof} The assertion (i) is trivial. 
For the proof of (ii) it suffices to show 
\begin{equation}
\psi_f^{\mer}( R 
\Gamma_{P^{-1}(0) \cup Q^{-1}(0)} ( \F ))
\simeq 0. 
\end{equation}
But this follows from (iii). 
Let us prove (iii). The problem being local, 
we may assume that $X= \CC^n$ 
and $x$ is the origin $0\in P^{-1}(0) 
\subset X= \CC^n$. 
We denote by $F_{t,0}$ the (usual) Milnor fiber 
of the projection $t\colon X \times 
\CC_{t}\to \CC_{t}$ at $0 \in X= \CC^n$. 
We will see later that $Ri_{f*}
( \F |_{X \setminus Q^{-1}(0)})$ is 
constructible (see the proof of Theorem~\ref{the-1}).
Therefore, by the basic fact for 
the nearby cycle functors 
(see e.g. \cite[Proposition~4.2.2]{Dimca}),  
we have an isomorphism
\begin{align}
H^{j} \psi_t( Ri_{f*}
( \F |_{X \setminus Q^{-1}(0)}) )_{0}
\simeq H^{j}(F_{t,0}; Ri_{f*}
( \F |_{X \setminus Q^{-1}(0)}))
\end{align}
which is compatible with the 
monodromy automorphisms on the both sides. 
Consider a Whitney stratification of 
$X \times \CC$ adapted to $Ri_{f*}
( \F |_{X \setminus Q^{-1}(0)})$ and refining 
the partition $X \times \CC = 
i_f(X \setminus Q^{-1}(0)) \sqcup 
(X \times \CC) \setminus i_f(X \setminus Q^{-1}(0))$. 
Then, for any $t \in \CC^*$ such that 
$0< |t| \ll 1$ the hyperplane $H_t= X \times \{ t \} 
\subset X \times \CC$ of $X \times \CC$ 
intersects the strata in it transversally 
on a neighborhood of $F_{t,0} \subset H_t$ 
(see e.g. \cite[Proposition 1.3]{M}). 
This implies that for the inclusion map 
\begin{equation}
i_{f,t}: i_f^{-1}(H_t) \simeq 
i_f(X \setminus Q^{-1}(0)) \cap H_t 
\hookrightarrow H_t
\end{equation}
we have an isomorphism 
\begin{equation}
\{ Ri_{f*}( \F |_{X \setminus Q^{-1}(0)})
\}|_{H_t} \simeq 
R(i_{f,t})_* ( \F |_{i_f^{-1}(H_t)}) 
\end{equation}
on the open subset $F_{t,0} \subset H_t$ of $H_t$. 
Since we have $i_{f,t}^{-1}(F_{t,0})=F_{0}$, 
the cohomology group 
$H^{j}(F_{t,0}; Ri_{f*}
( \F |_{X \setminus Q^{-1}(0)}))$ 
is isomorphic to $H^{j}(F_{0}; \F )$. 
\end{proof}

From now on, we shall prove that the functor 
\begin{equation}
\psi_f^{\mer}( \cdot ) : 
\Db(X) \longrightarrow \Db(X)
\end{equation}
thus defined preserves the constructibility and 
the perversity (up to some shift). For the 
closed embedding 
\begin{equation}
k_f: X \setminus Q^{-1}(0) \hookrightarrow 
(X \setminus Q^{-1}(0)) \times \CC_t
\end{equation}
defined by $x \mapsto (x,f(x))$ and the 
inclusion map $j_f: (X \setminus Q^{-1}(0)) \times \CC_t 
\hookrightarrow X \times \CC_t$ we have 
$i_f=j_f \circ k_f$ and hence an isomorphism 
\begin{equation}
Ri_{f*}( \F |_{X \setminus Q^{-1}(0)}) \simeq 
Rj_{f*}(Rk_{f*}( \F |_{X \setminus Q^{-1}(0)}) ). 
\end{equation}
Moreover, if $\F \in \Db(X)$ is constructible 
(resp. perverse), then $Rk_{f*}( \F |_{X \setminus Q^{-1}(0)}) 
\in \Db ( (X \setminus Q^{-1}(0)) \times \CC_t )$ is 
constructible (resp. perverse). However the functor 
\begin{equation}
Rj_{f*}  : 
\Db(  (X \setminus Q^{-1}(0)) \times \CC_t  ) 
\longrightarrow \Db( X \times \CC_t )
\end{equation}
does not preserve the constructibility 
(resp. perversity) in general. Nevertheless, 
we can overcome this difficulty as follows. 

\begin{theorem}\label{the-1} 
\begin{enumerate}
\item[\rm{(i)}] If $\F \in \Db(X)$ is constructible, then 
$\psi_f^{\mer}( \F ) \in \Db(X)$ is also constructible. 
\item[\rm{(ii)}] If $\F \in \Db(X)$ is perverse, then 
$\psi_f^{\mer}( \F )[-1] \in \Db(X)$ is also perverse. 
\end{enumerate}
\end{theorem}

\begin{proof} Assume that $\F \in \Db(X)$ is 
constructible. Define a hypersurface $W$ of 
$X \times \CC_t$ by 
\begin{equation}
W= \{ (x,t) \in X \times \CC \ | \ 
P(x)-tQ(x)=0 \} 
\end{equation}
and let $\rho : W \rightarrow X$ be the 
restriction of the first projection 
$X \times \CC_t \rightarrow X$ to it. 
Then $\rho$ induces an isomorphism 
\begin{equation}
\rho^{-1} (X \setminus Q^{-1}(0)) 
\simto X \setminus Q^{-1}(0)
\end{equation}
and $\rho^{-1} (X \setminus Q^{-1}(0))$ 
is nothing but the graph 
\begin{equation}
\{ (x, f(x)) \in (X \setminus Q^{-1}(0)) 
\times \CC \ | \ x \in X \setminus Q^{-1}(0) \} 
\end{equation}
of $f: X \setminus Q^{-1}(0) \rightarrow \CC$. 
In this way, we identify $X \setminus Q^{-1}(0)$ 
and the open subset $\rho^{-1} (X \setminus Q^{-1}(0))$ 
of $W$. Let 
\begin{equation}
\iota_f: X \setminus Q^{-1}(0) \simeq 
\rho^{-1} (X \setminus Q^{-1}(0)) 
\hookrightarrow W 
\end{equation}
and $i_W: W \hookrightarrow X \times \CC_t$ 
be the inclusion maps. Then for the 
constructible sheaf $\F \in \Db(X)$ 
we have an isomorphism 
\begin{equation}
Ri_{f*}( \F |_{X \setminus Q^{-1}(0)}) \simeq 
i_{W*}(R \iota_{f*}( \F |_{X \setminus Q^{-1}(0)}) ). 
\end{equation}
Moreover, by the Cartesian diagram 
\begin{equation}
\begin{CD}
X \setminus Q^{-1}(0)  @>{\iota_f}>> W 
\\
@V{k_f}VV       @VV{i_W}V
\\
(X \setminus Q^{-1}(0)) \times \CC @>>{j_f}> X \times \CC 
\end{CD}
\end{equation}
we obtain isomorphisms 
\begin{align}
Rj_{f_*}j_f^{-1}(i_{W*} \rho^{-1} \F )
& \simeq 
Rj_{f_*}k_{f*} \iota_f^{-1} \rho^{-1} \F 
\nonumber 
\\
 & \simeq 
i_{W*}(R \iota_{f*}( \F |_{X \setminus Q^{-1}(0)}) ). 
\end{align}
Then we obtain the constructibility of 
\begin{align}
Ri_{f*}( \F |_{X \setminus Q^{-1}(0)}) 
& \simeq 
Rj_{f_*}j_f^{-1}(i_{W*} \rho^{-1} \F )
\nonumber 
\\
 & \simeq 
R \shom_{\CC_{X \times \CC}}( 
\CC_{(X \setminus Q^{-1}(0)) \times \CC}, 
i_{W*} \rho^{-1} \F )
\end{align}
by \cite[Theorem 8.5.7 (ii)]{K-S}. If moreover 
$\F \in \Db(X)$ is perverse, then the 
perversity of $Rj_{f_*}j_f^{-1}(i_{W*} \rho^{-1} \F )
\in \Dbc (X \times \CC_t)$ follows from 
\cite[Proposition 10.3.17 (i)]{K-S} on the Stein map 
$j_f: (X \setminus Q^{-1}(0)) \times \CC_t 
\hookrightarrow X \times \CC_t$ 
(see also the paragraph below 
\cite[Corollary 5.2.17]{Dimca} and Sch\"urmann 
\cite[page 410]{Sch}). 
Finally, by applying the t-exact functor 
$\psi_t( \cdot ) [-1] : \Dbc (X \times \CC_t) 
\longrightarrow \Dbc (X)$ 
to it, we obtain the assertions. 
\end{proof}

By this theorem we obtain a functor 
\begin{equation}
\psi_f^{\mer}( \cdot ) : 
\Dbc (X) \longrightarrow \Dbc (X). 
\end{equation}
Moreover by its proof, for $\F \in \Dbc (X)$ 
there exists a natural morphism 
\begin{equation}
\F_{P^{-1}(0)} \longrightarrow \psi_f^{\mer}( \F ). 
\end{equation}

\begin{remark}\label{rem-1} 
Assume that the meromorphic function $f(x)= \frac{P(x)}{Q(x)}$ 
is holomorphic on a neighborhood of a point 
$x \in X$ i.e. there exists a holomorphic function 
$g(x)$ defined on a neighborhood of 
$x \in X$ such that $P(x)=Q(x) \cdot g(x)$ on it. 
Then by identifying $f(x)$ with the 
holomorphic function $g(x)$, 
we have an isomorphism 
\begin{equation}
\psi_f^{\mer}( \F )_x \simeq 
\psi_f ( R 
\Gamma_{X \setminus Q^{-1}(0)} ( \F ))_x
\end{equation}
for the classical (holomorphic) nearby cycle 
functor $\psi_f( \cdot )$. This implies that 
even if $f(x)$ is holomorphic on 
a neighborhood of $x \in X$ we do not 
have an isomorphism 
\begin{equation}
\psi_f^{\mer}( \F )_x \simeq 
\psi_f ( \F )_x
\end{equation}
in general. 
\end{remark}

The following useful result is an analogue 
for $\psi_f^{\mer}( \cdot )$ of the 
classical one for $\psi_f( \cdot )$ 
(see e.g. \cite[Proposition 4.2.11]{Dimca} 
and \cite[Exercise VIII.15]{K-S} etc.). 

\begin{proposition}\label{PDI} 
Let $\pi : Y \longrightarrow X$ be a proper 
morphism of complex manifolds and 
$f \circ \pi$ a meromorphic function on $Y$ defined by 
\begin{equation}
f \circ \pi = \frac{P \circ \pi}{Q \circ \pi}. 
\end{equation}
Then for $\G \in \Db(Y)$ 
there exists an isomorphism 
\begin{equation}
\psi_{f}^{\mer} ( R \pi_* \G ) \simeq 
R \pi_* \psi_{f \circ \pi}^{\mer} ( \G ). 
\end{equation}
If moreover $\pi$ induces an isomorphism 
\begin{equation}
Y \setminus \pi^{-1}(P^{-1}(0) \cup Q^{-1}(0)) 
\simto X \setminus (P^{-1}(0) \cup Q^{-1}(0)), 
\end{equation}
then for $\F \in \Db(X)$ 
there exists an isomorphism 
\begin{equation}
\psi_{f}^{\mer} ( \F ) \simeq 
R \pi_* \psi_{f \circ \pi}^{\mer} ( \pi^{-1} \F ). 
\end{equation}
\end{proposition}

\begin{proof} 
Let $\pi^{\circ}: Y \setminus (Q \circ \pi)^{-1}(0)
\rightarrow 
X \setminus Q^{-1}(0)$ be the restriction of $\pi$ 
to $Y \setminus (Q \circ \pi)^{-1}(0)$. Then by 
the commutative diagram 
\begin{equation}
\begin{CD}
Y \setminus (Q \circ \pi)^{-1}(0)  @>{i_{f \circ \pi}}>> Y \times \CC_t
\\
@V{\pi^{\circ}}VV       @VV{\pi \times \id_{\CC_t}}V
\\
X \setminus Q^{-1}(0) @>>{i_f}> X \times \CC_t 
\end{CD}
\end{equation}
for $\G \in \Db(Y)$ 
there exist isomorphisms 
\begin{align}
  Ri_{f*}
( R \pi_* \G |_{X \setminus Q^{-1}(0)}) & 
\simeq 
Ri_{f*}
( R \pi^{\circ}_* (\G |_{Y \setminus (Q \circ \pi)^{-1}(0)}))
\nonumber 
\\
& \simeq  
R ( \pi \times \id_{\CC_t} )_* Ri_{f \circ \pi *}
(\G |_{Y \setminus (Q \circ \pi)^{-1}(0)}). 
\end{align}
Then by \cite[Proposition 4.2.11]{Dimca} 
and \cite[Exercise VIII.15]{K-S} we obtain 
isomorphisms 
\begin{align}
\psi_{f}^{\mer} ( R \pi_* \G ) 
& = \psi_t( Ri_{f*}
( R \pi_* \G |_{X \setminus Q^{-1}(0)}) ) 
\nonumber 
\\ 
& \simeq 
\psi_t( R ( \pi \times \id_{\CC_t} )_* Ri_{f \circ \pi *}
(\G |_{Y \setminus (Q \circ \pi)^{-1}(0)}))
\nonumber 
\\
& \simeq 
R \pi_* \psi_t( Ri_{f \circ \pi *}
(\G |_{Y \setminus (Q \circ \pi)^{-1}(0)})) 
= R \pi_* \psi_{f \circ \pi}^{\mer} ( \G ). 
\end{align}
Assume now that $\pi$ induces an isomorphism 
\begin{equation}
Y \setminus \pi^{-1}(P^{-1}(0) \cup Q^{-1}(0)) 
\simto X \setminus (P^{-1}(0) \cup Q^{-1}(0)). 
\end{equation}
Then for $\F \in \Db(X)$ 
we have an isomorphism 
\begin{equation}
R \Gamma_{X \setminus (P^{-1}(0) \cup Q^{-1}(0))} 
( \F ) \simto 
R \Gamma_{X \setminus (P^{-1}(0) \cup Q^{-1}(0))} 
(R \pi_* \pi^{-1} \F ). 
\end{equation}
By Lemma \ref{lem-1} (ii) we thus obtain 
isomorphisms 
\begin{align}
& \psi_{f}^{\mer} ( \F ) 
\simeq 
\psi_{f}^{\mer} ( 
R \Gamma_{X \setminus (P^{-1}(0) \cup Q^{-1}(0))} 
( \F )) 
\nonumber 
\\ 
& \simeq 
\psi_{f}^{\mer} ( R \pi_* \pi^{-1} \F ) 
\simeq 
R \pi_* \psi_{f \circ \pi}^{\mer} ( \pi^{-1} \F ). 
\end{align}
This completes the proof. 
\end{proof} 

For the meromorphic function $f(x)= \frac{P(x)}{Q(x)}$ 
and $\F \in \Db(X)$ we set also 
\begin{equation}
\psi_f^{\merc}( \F ):= \psi_t( Ri_{f!}
( \F |_{X \setminus Q^{-1}(0)}) ) 
\in \Db(X). 
\end{equation}
(see Raibaut \cite{Raibaut-2}). 
We call it the meromorphic 
nearby cycle sheaf 
with compact support of $\F$ along $f$. 
Then we obtain a functor 
\begin{equation}
\psi_f^{\merc}( \cdot ) : 
\Db(X) \longrightarrow \Db(X)
\end{equation}
which satisfies the properties similar to 
the ones in  Lemma \ref{lem-1} (i) (ii), 
Theorem \ref{the-1} and Proposition 
\ref{PDI}. Moreover if  $f$ 
is holomorphic on a neighborhood of a point 
$x \in X$, then we have an isomorphism 
\begin{equation}
\psi_f^{\merc}( \F )_x \simeq 
\psi_f ( \F_{X \setminus Q^{-1}(0)} )_x
\end{equation}
for the classical (holomorphic) nearby cycle 
functor $\psi_f( \cdot )$. 
However the isomorphism in Lemma \ref{lem-1} (iii) 
does not hold for $\psi_f^{\merc}( \F )$. 
This implies that the natural morphism 
\begin{equation}
\psi_f^{\merc}( \F ) \longrightarrow 
\psi_f^{\mer}( \F )
\end{equation}
is not an isomorphism in general. 
In fact, in \cite{Raibaut-2} Raibaut introduced 
a functor which is identical to 
our $\psi_f^{\merc}( \cdot )$. 

From now on, we restrict ourselves to the case $k= \CC$. 
For a point $x \in X$ and $\F \in \Dbc(X)$ let 
\begin{equation}
\Phi( \F )_{j,x}: H^j \psi_f^{\mer}( \F )_x 
 \simto H^j \psi_f^{\mer}( \F )_x  \qquad 
(j \in \ZZ )
\end{equation}
be the monodromy 
automorphisms of $H^j \psi_f^{\mer}( \F )$ 
at $x$. We define 
the monodromy zeta function $\zeta( \F )_{f,x}(t) \in \CC (t)$ 
of $f$ for $\F$ at $x \in X$ by 
\begin{equation}
\zeta( \F )_{f,x}(t) = \prod_{j \in \ZZ} 
\Bigl\{ {\rm det}( \id - t \Phi( \F )_{j,x}) \Bigr\}^{(-1)^j}  
\in \CC (t). 
\end{equation}
Then we obtain the following sheaf-theoretical 
reformulation of \cite[Theorem 1]{G-L-M}, 
which is also an analogue 
for meromorphic functions of 
\cite[Propositions 5.2 and 5.3]{M-T-1}. 

\begin{proposition}\label{ProI} 
(cf. \cite[Theorem 1]{G-L-M}) 
Assume that $X= \CC^n$, the hypersurface 
$P^{-1}(0) \cup Q^{-1}(0)$ of $X$ is 
a normal crossing divisor $\{ x \in X \ | \ 
x_1x_2 \cdots x_r=0 \}$ for some $1 \leq r \leq n$ 
and $f(x)=x_1^{m_1} x_2^{m_2} \cdots x_r^{m_r}$ 
($m_i \in \ZZ$). By the inclusion map 
$j: X \setminus (P^{-1}(0) \cup Q^{-1}(0)) 
\hookrightarrow X$ set 
\begin{equation}
\F 
= Rj_* ( \CC_{X \setminus (P^{-1}(0) \cup Q^{-1}(0))} ) 
\in \Dbc(X). 
\end{equation}
Then, if $r=1$ and $m_1>0$ we 
have $\zeta( \F )_{f,0}(t)=1-t^{m_1}$. 
Otherwise, we have $\zeta( \F )_{f,0}(t)=1$. 
\end{proposition}

\begin{proof} 
As in the proof of \cite[Theorem 3.6]{M-T-2} we 
construct towers of blow-ups of $X$ over 
the normal crossing divisor 
$P^{-1}(0) \cup Q^{-1}(0)$ to eliminate the 
points of indeterminacy of $f$. Then we 
obtain a proper morphism $\pi : Y \longrightarrow X$ 
of complex manifolds which induces an isomorphism 
\begin{equation}
Y \setminus \pi^{-1}(P^{-1}(0) \cup Q^{-1}(0)) 
\simto X \setminus (P^{-1}(0) \cup Q^{-1}(0)). 
\end{equation}
If $r \geq 2$, then 
by calculating  
$\zeta( \pi^{-1} \F )_{f \circ \pi ,y}(t) \in \CC (t)$ 
at each point $y$ of $\pi^{-1}(0)$ 
the assertion follows from 
\cite[page 170-173]{Dimca} (see also 
\cite{Sch} and \cite[Proposition 2.9]{M-T-2}) 
and Proposition \ref{PDI}. 
\end{proof} 

For a (shifted) perverse sheaf $\F \in \Dbc(X)$ on $X$ 
and $\lambda \in \CC$, let 
\begin{equation}
\Phi( \F ):  \psi_f^{\mer}( \F ) \simto \psi_f^{\mer}( \F )  
\end{equation}
be the monodromy automorphism of $\psi_f^{\mer}( \F )$ 
and by taking $N \gg 0$ set 
\begin{equation}
\psi_{f, \lambda}^{\mer}( \F ):= {\rm Ker} 
\Bigl[ 
\bigl( \lambda \cdot {\rm id} - \Phi( \F ) \bigr)^N: 
\psi_f^{\mer}( \F ) \longrightarrow \psi_f^{\mer}( \F )  
\Big]  \in \Dbc(X), 
\end{equation}
where the right hand side is the kernel in 
the abelian category of (shifted) perverse sheaves. 
We can define $\psi_{f, \lambda}^{\mer}( \F )$ 
also when $\F$ is not perverse but constructible. 
For a more precise account, see e.g. \cite[Remark 4.2.5]{Dimca}. 
We call $\psi_{f, \lambda}^{\mer}( \F )$ 
the generalized eigenspace 
of $\psi_{f}^{\mer}( \F )$ for the eigenvalue $\lambda$. 
Then we have a decomposition 
\begin{equation}
\psi_{f}^{\mer}( \F ) \simeq 
\bigoplus_{\lambda \in \CC} 
\psi_{f, \lambda}^{\mer}( \F ). 
\end{equation}
Similarly we can define also $\psi_{f, \lambda}^{\merc}( \F )$. 
The following result will be used in the proof 
of Proposition \ref{pro-1}. 

\begin{proposition}\label{ProII} 
Assume that $X= \CC^n$ and the meromorphic 
function $f(x)= \frac{P(x)}{Q(x)}$ is defined 
by $P(x) =x_1^{m_1} x_2^{m_2} \cdots x_k^{m_k}$ 
($m_i \in \ZZ$, $m_i>0$) for some $1 \leq k \leq n-1$ 
and $Q(x)=x_n$. Then for any $\lambda \not= 1$ 
the natural morphism 
\begin{equation}
\psi_{f, \lambda}^{\merc}( \CC_X )_0 \longrightarrow 
\psi_{f, \lambda}^{\mer}( \CC_X )_0
\end{equation}
is an isomorphism. 
\end{proposition}

\begin{proof} 
Note that the hypersurface 
$P^{-1}(0) \cup Q^{-1}(0)$ of $X$ is 
a normal crossing divisor $\{ x \in X \ | \ 
x_1x_2 \cdots x_k x_n=0 \}$. 
As in the proof of \cite[Theorem 3.6]{M-T-2} we 
construct towers of blow-ups of $X$ over 
the normal crossing divisor 
$P^{-1}(0) \cup Q^{-1}(0)$ to eliminate the 
points of indeterminacy of $f$. Then we 
obtain a proper morphism $\pi : Y \longrightarrow X$ 
of complex manifolds which induces an isomorphism 
\begin{equation}
Y \setminus \pi^{-1}(P^{-1}(0) \cup Q^{-1}(0)) 
\simto X \setminus (P^{-1}(0) \cup Q^{-1}(0)). 
\end{equation}
By Proposition \ref{PDI} and its analogue for 
$\psi_f^{\merc}( \cdot )$ we obtain isomorphisms 
\begin{align}
\psi_{f, \lambda}^{\merc}( \CC_X ) & \simeq R \pi_* 
\psi_{f \circ \pi, \lambda}^{\merc}( \CC_Y ),
\\
\psi_{f, \lambda}^{\mer}( \CC_X ) & \simeq R \pi_* 
\psi_{f \circ \pi, \lambda}^{\mer}( \CC_Y )
\end{align}
for any $\lambda \in \CC$. Set 
$D= \pi^{-1}(P^{-1}(0) \cup Q^{-1}(0))$ and 
let $j: Y \setminus D \hookrightarrow Y$ be the 
inclusion map. Then by Remark \ref{rem-1} and 
its analogue for 
$\psi_f^{\merc}( \cdot )$ we have isomorphisms 
\begin{align}
\psi_{f \circ \pi, \lambda}^{\merc}( \CC_Y ) & 
\simeq \psi_{f \circ \pi, \lambda}( j_! \CC_{Y \setminus D}), 
\\
\psi_{f \circ \pi, \lambda}^{\mer}( \CC_Y ) & 
\simeq \psi_{f \circ \pi, \lambda}( Rj_* \CC_{Y \setminus D})
\end{align}
for any $\lambda \in \CC$. Hence it suffices to prove 
that the natural morphism 
\begin{equation}
R \Gamma ( \pi^{-1}(0); \psi_{f \circ \pi, 
\lambda}( j_! \CC_{Y \setminus D}))
\longrightarrow 
R \Gamma ( \pi^{-1}(0); 
\psi_{f \circ \pi, \lambda}( Rj_* \CC_{Y \setminus D}))
\end{equation}
is an isomorphism for any $\lambda \not= 1$. 
By the distinguished triangle 
\begin{equation}
j_! \CC_{Y \setminus D} \longrightarrow 
Rj_* \CC_{Y \setminus D} 
\longrightarrow (Rj_* \CC_{Y \setminus D})_{D} 
\overset{+1}{\longrightarrow} 
\end{equation}
we have only to show the vanishing 
\begin{equation}
R \Gamma ( \pi^{-1}(0); \psi_{f \circ \pi, \lambda}(
(Rj_* \CC_{Y \setminus D})_{D} 
)) \simeq 0 
\end{equation}
for any $\lambda \not= 1$. From now on, 
assume that $\lambda \not= 1$. First let us treat 
the simplest case $k=1$. In this case, the 
normal crossing divisor $D= \pi^{-1}(P^{-1}(0) \cup Q^{-1}(0))$ 
in $Y$ has $m_1+2$ irreducible components $D_i$ 
($-1 \leq i \leq m_1$) such that we have 
${\rm ord}_{D_i} (f \circ \pi )=i$. 
See the proof of \cite[Theorem 3.6]{M-T-2} for the details. 
Note that $D_{-1}$ 
is noting but the proper transform of the pole set 
$Q^{-1}(0)= \{ x \in X \ | \ x_n=0 \}$ of $f$ in $Y$. 
Since the support of $(Rj_* \CC_{Y \setminus D})_{D}$ 
is contained in $D$, that of its nearby cycle 
$\psi_{f \circ \pi}(
(Rj_* \CC_{Y \setminus D})_{D})$ is contained in 
$D_1 \cap D_0$. But by our assumption $\lambda \not= 1$ 
we have $\psi_{f \circ \pi, \lambda}(
(Rj_* \CC_{Y \setminus D})_{D}) \simeq 0$ also on 
$D_1 \cap D_0$. Next consider the case $k=2$. Let 
$\pi_1 : Y_1 \longrightarrow X$ be the blow-up of 
$X$ over the set $\{ x \in X \ | \ x_1=x_n=0 \} \subset 
I(f)$ in the first step of the 
construction of $\pi : Y \longrightarrow X$. 
We define divisors $D_i \subset Y_1$ 
($-1 \leq i \leq m_1$) as in the case $k=1$. 
Now let $K \subset Y_1$ be the 
proper transform of 
$\{ x \in X \ | \ x_2=0 \} \subset X$ in $Y_1$. 
Then $f \circ \pi_1$ still has some points of 
indeterminacy in the set $D_{-1} \cap K$. So 
we construct a tower of blow-ups over it 
until we get a morphism $\pi_2: Y \longrightarrow Y_1$ 
such that $\pi = \pi_1 \circ \pi_2$. 
See the proof of \cite[Theorem 3.6]{M-T-2} for the details. 
Let $E_i \subset Y$ ($-1 \leq i \leq m_2$) 
be the exceptional divisors of $\pi_2$ such that 
${\rm ord}_{E_i} (f \circ \pi )=i$. Denote the 
proper transform of $D_0 \subset Y_1$ in $Y$ by $H$. 
Then (on a neighborhood of $E= \pi_2^{-1}(D_{-1} \cap K)$ 
in $Y$) we have 
\begin{equation}
(f \circ \pi )^{-1}(0)= 
E_1 \cup E_2 \cup \cdots \cup E_{m_2}
\end{equation}
and the support of $(Rj_* \CC_{Y \setminus D})_{D}$ 
is contained in $E_{-1} \cup \cdots \cup E_{m_2} \cup H$. 
For $\lambda \not= 1$ this implies that 
the support of 
$\psi_{f \circ \pi, \lambda}(
(Rj_* \CC_{Y \setminus D})_{D})$ is contained in 
$(E_1 \cup \cdots \cup E_{m_2}) \cap H$. Moreover 
by truncation functors, it suffices to prove the 
vanishing 
\begin{equation}
R \Gamma ( \pi^{-1}(0); 
\psi_{f \circ \pi, \lambda}( \CC_H ) 
) \simeq 0. 
\end{equation}
But this follows from the primitive decompositons 
(of the graded pieces w.r.t. the weight filtration) of 
the nearby cycle sheaf $\psi_{f \circ \pi, \lambda}( \CC_H )$ 
(see e.g. \cite{D-L-1} etc. for the details). 
Indeed, for any $1 \leq i \leq m_2-1$ the restriction of 
$\psi_{f \circ \pi, \lambda}( \CC_H )$ to the subset 
\begin{equation}
\pi^{-1}(0) \cap \{ E_i \setminus (E_{i-1} \cup E_{i+1}) \} 
\simeq \CC^*
\end{equation}
of $\pi^{-1}(0)$ is zero or a non-trivial local 
system of rank one. Moreover by our assumption 
$\lambda \not= 1$ we never have 
the condition $\lambda^i= \lambda^{i+1}=1$. 
This implies that the restriction of 
$\psi_{f \circ \pi, \lambda}( \CC_H )$ to 
$E_i \cap E_{i+1}$ is zero. Similarly, we can prove 
the assertion for any $1 \leq k \leq n-1$. 
This completes the proof. 
\end{proof}

\section{Milnor monodromies of meromorphic functions}\label{sec:s3}

In this section, by using the meromorphic nearby cycle functors 
introduced in Section \ref{sec:s2} we 
study Milnor monodromies of meromorphic functions. 
Let us consider the meromorphic function $f(x)= \frac{P(x)}{Q(x)}$ 
in Section \ref{sec:s2}. The problem being local, we may 
assume that $X= \CC^n$ and $0 \in I(f)=P^{-1}(0) \cap Q^{-1}(0)$. 
Let $F_0$ be the Milnor fiber of $f$ at the origin 
$0 \in X= \CC^n$ and 
\begin{equation}
\Phi_{j,0}: H^j(F_0; \CC ) \simto H^j(F_0; \CC ) \qquad 
(j \in \ZZ )
\end{equation}
its Milnor monodromy operators. Then we define 
the monodromy zeta function $\zeta_{f,0}(t) \in \CC (t)$ 
of $f$ at the origin $0 \in X= \CC^n$ by 
\begin{equation}
\zeta_{f,0}(t) = \prod_{j \in \ZZ} 
\Bigl\{ {\rm det}( \id - t \Phi_{j,0}) \Bigr\}^{(-1)^j}  \in \CC (t). 
\end{equation}

By Propositions \ref{PDI} and \ref{ProI} we can reprove 
the following result of 
Gusein-Zade, Luengo and Melle-Hern\'andez 
\cite[Theorem 1]{G-L-M}, \cite[Theorem 1.19]{G-L-M-new-new}, 
which is an analogue for meromorphic functions 
of A'Campo's formula in \cite{Acampo}. 

\begin{theorem}\label{the-5} 
(Gusein-Zade, Luengo and Melle-Hern\'andez 
\cite[Theorem 1]{G-L-M}, \cite[Theorem 1.19]{G-L-M-new-new}) 
Let $\pi: Y \rightarrow X= \CC^n$ be a 
resolution of singularities of $P^{-1}(0) \cup 
Q^{-1}(0)$ which induces an isomorphism 
\begin{equation}
Y \setminus \pi^{-1}(P^{-1}(0) \cup Q^{-1}(0)) 
\simto X \setminus (P^{-1}(0) \cup Q^{-1}(0))
\end{equation}
such that $\pi^{-1}(0)$ and 
$D=\pi^{-1}(P^{-1}(0) \cup Q^{-1}(0))$ are 
strict normal crossing divisors in $Y$. 
Let $\cup_{i=1}^k D_i$ be the irreducible 
decomposition of $D=\pi^{-1}(P^{-1}(0) \cup Q^{-1}(0))$ 
such that $\pi^{-1}(0)= \cup_{i=1}^r D_i$ for 
some $1 \leq r \leq k$. For $1 \leq i \leq r$ set 
\begin{equation}
D_i^{\circ}= D_i \setminus (\cup_{j \not= i} D_j)
\end{equation}
and 
\begin{equation}
m_i= {\rm ord}_{D_i} (P \circ \pi ) - 
{\rm ord}_{D_i} (Q \circ \pi ) \in \ZZ. 
\end{equation}
Then we have 
\begin{equation}
\zeta_{f,0}(t) = 
\prod_{i: m_i >0} 
(1- t^{m_i})^{\chi ( D_i^{\circ} )}. 
\end{equation}
\end{theorem}

For $m \geq 1$ we define the Lefschetz number 
$\Lambda (m)_{f,0} \in \ZZ$ of the meromorphic 
function $f$ at the origin $0 \in X= \CC^n$ by 
\begin{equation}
\Lambda (m)_{f,0}= 
\dsum_{j \in \ZZ} (-1)^j {\rm tr} 
\bigl\{
\Phi_{j,0}^m : H^j(F_0; \CC ) \simto H^j(F_0; \CC )
\bigr\}. 
\end{equation}
Then as in \cite{AC} we obtain the following 
corollary (see also e.g. \cite[Remark 3.2]{M-T-2}). 

\begin{corollary}\label{cor-111} 
In the situation of Theorem \ref{the-5}, for 
any $m \geq 1$ we have 
\begin{equation}
\Lambda (m)_{f,0}= 
\dsum_{i: m_i>0, m_i|m} 
\chi ( D_i^{\circ} ) \cdot m_i. 
\end{equation}
\end{corollary}

From now, let us prove Theorem \ref{the-3}. 

\begin{proof} 
By Theorem \ref{the-1} $\psi_f^{\mer}( \CC_X [n] )[-1] 
\in \Db (X)$ is a perverse sheaf. The same is true 
also for its $\lambda$-part $\psi_{f, \lambda}^{\mer}
( \CC_X [n] )[-1] \in \Db (X)$. 
Let $\pi_0: \widetilde{X} \rightarrow X= \CC^n$ be a 
resolution of singularities of $P^{-1}(0)$ and 
$Q^{-1}(0)$ which induces an isomorphism 
$\widetilde{X} \setminus \pi_0^{-1}( \{ 0 \} ) 
\simto X \setminus \{ 0 \}$. Let $\widetilde{P^{-1}(0)}$ 
and $\widetilde{Q^{-1}(0)}$ be the (smooth) proper 
transforms of $P^{-1}(0)$ and 
$Q^{-1}(0)$ in $\widetilde{X}$ respectively. 
We may assume that they intersect transversally. 
Now let $\pi_1: Y \rightarrow \widetilde{X}$ be 
the blow-up of $\widetilde{X}$ along 
$\widetilde{P^{-1}(0)} \cap \widetilde{Q^{-1}(0)}$ 
and set $\pi := \pi_0 \circ \pi_1:  Y \rightarrow X$. 
Then by Propostion \ref{PDI} there exists an isomorphism 
\begin{equation}
\psi_{f, \lambda}^{\mer}
( \CC_X [n] )[-1] \simeq 
R \pi_* \psi_{f \circ \pi, \lambda}^{\mer} ( \CC_Y [n] )[-1]. 
\end{equation}
By the construction of $\pi$, we can show 
that for $\lambda \not=1$ the support of 
$\psi_{f \circ \pi, \lambda}^{\mer} ( \CC_Y [n] )[-1] 
\in \Db (Y)$ is contained in $\pi^{-1}( \{ 0 \} )
\subset Y$. 
Indeed, let $E \subset Y$ be the exceptional 
divisor of the blow-up $\pi_1: Y \rightarrow \widetilde{X}$ 
and denote the (smooth) proper transform of 
$\widetilde{P^{-1}(0)}$ (resp. $\widetilde{Q^{-1}(0)}$) 
in $Y$ by $D_P$ (resp. $D_Q$). 
Then we have $D_P \cap D_Q= \emptyset$ 
and $D_P \cup D_Q \cup E$ is a normal 
crossing divisor in $Y$. Moreover 
on $Y \setminus \pi^{-1}( \{ 0 \} )$ 
(i.e. outside $\pi^{-1}( \{ 0 \} )$) 
the meromorphic function $f \circ \pi$ on $Y$ 
takes the value $0$ only on $D_P$. 
This implies that on $Y \setminus \pi^{-1}( \{ 0 \} )$ 
the support of the perverse sheaf 
$\psi_{f \circ \pi}^{\mer} ( \CC_Y [n] )[-1] 
\in \Db (Y)$ is contained in $D_P$. 
Since the divisor $D_P \cup D_Q \cup E$ is normal 
crossing, we can easily see also that 
there exists an isomorphism 
\begin{equation}
\psi_{f \circ \pi}^{\mer} ( \CC_Y [n] )[-1] 
\simeq \psi_{f \circ \pi, 1}^{\mer} ( \CC_Y [n] )[-1]
\end{equation}
on $Y \setminus \pi^{-1}( \{ 0 \} )$. 
Hence the support of the perverse 
sheaf $\psi_{f, \lambda}^{\mer}
( \CC_X [n-1] ) \in \Db (X)$ is contained 
in the origin $\{ 0 \} \subset X$. 
Then by \cite[Proposition 8.1.22]{H-T-T} 
we obtain the concentration 
\begin{equation}
H^j \psi_{f, \lambda}^{\mer}
( \CC_X [n-1] )_0 \simeq 
H^{j+n-1}(F_0; \CC )_{\lambda} \simeq 0 
\qquad (j \not= 0). 
\end{equation}
\end{proof} 

By Theorems \ref{the-5} and \ref{the-3} 
we obtain the following result. 

\begin{corollary}\label{cor-1} 
Assume that the meromorphic function 
$f(x)= \frac{P(x)}{Q(x)}$ satisfies the 
conditions in Theorem \ref{the-3}. 
Then in the notations of Theorem \ref{the-5}, 
for any $\lambda \not=1$ 
the multiplicity of the eigenvalue 
$\lambda$ in $\Phi_{n-1, 0}$ is 
equal to that of the factor $t- \lambda$ 
in the rational function 
\begin{equation}
\prod_{i: m_i >0} 
(t^{m_i}-1)^{(-1)^{n-1} \chi ( D_i^{\circ} )} 
\in \CC (t). 
\end{equation}
\end{corollary}

\begin{definition}
Let $g(x)=\sum_{v \in \ZZ^n} c_vx^v$ ($c_v\in \CC$) be a 
Laurent polynomial on the algebraic torus 
$T=(\CC^*)^n$. 
\begin{enumerate}
\item[\rm{(i)}] We call the convex hull of 
$\supp(g):=\{v\in \ZZ^n \ | \ c_v\neq 0\} 
\subset \ZZ^n \subset \RR^n$ in $\RR^n$ the 
Newton polytope of $g$ and denote it by $NP(g)$.
\item[\rm{(ii)}] If $g$ is a polynomial, 
we call the convex hull of $\cup_{v \in \supp(g)} 
(v + \RR_+^n)$ in $\RR_+^n$ the 
Newton polyhedron of $g$ at 
the origin $0 \in \CC^n$ and denote it by $\Gamma_+(g)$.
\item[\rm{(iii)}] For a face $\gamma \prec NP(g)$ of $NP(g)$, 
we define the $\gamma$-part 
$g^{\gamma}$ of $g$ by 
$g^{\gamma}(x):=\sum_{v \in \gamma} c_vx^v$. 
\end{enumerate}
\end{definition}

Let $\Gamma_+(P), \Gamma_+(Q) \subset \RR^n$ be 
the Newton polyhedra of $P$ and $Q$ at 
the origin $0 \in \CC^n$ and 
\begin{equation}
\Gamma_+(f)= \Gamma_+(P) + \Gamma_+(Q) 
\end{equation}
their Minkowski sum. From now, we recall 
Bernstein-Khovanskii-Kushnirenko's 
theorem \cite{Khovanskii}. 
Let $\Delta \subset \RR^n$ be a 
lattice polytope 
in $\RR^n$. For an element $u \in \RR^n$ of 
(the dual vector space of) $\RR^n$ we define the 
supporting face $\gamma_u \prec  \Delta$ 
of $u$ in $ \Delta$ by 
\begin{equation}
\gamma_u = \left\{ v \in \Delta \ | \ 
\langle u , v \rangle 
= 
\min_{w \in \Delta } 
\langle u ,w \rangle \right\}, 
\end{equation}
where for $u=(u_1,\ldots,u_n)$ 
and $v=(v_1,\ldots, v_n)$ we set 
$\langle u,v\rangle =\sum_{i=1}^n u_iv_i$. 
For a face $\gamma$ of $\Delta$ set 
\begin{equation}
\sigma (\gamma) = \overline{ \{ u \in \RR^n \ | \ 
\gamma_u = \gamma   \} } \subset \RR^n . 
\end{equation}
\noindent Then  $\sigma (\gamma )$ 
is an $(n- \dim \gamma )$-dimensional 
rational convex polyhedral 
cone in $\RR^n$. Moreover 
the family $\{ \sigma (\gamma ) \ | \ 
\gamma \prec  \Delta \}$ of cones in $\RR^n$ 
thus obtained is a subdivision of $\RR^n$. 
We call it the dual subdivision of $\RR^n$ by 
$\Delta$. If $\dim \Delta =n$ it 
satisfies the axiom 
of fans (see \cite{Fulton} and 
\cite{Oda} etc.). We call it the dual fan of 
$\Delta$. More generally, let $\Delta_1, \ldots, \Delta_p  
\subset \RR^n$ be lattice polytopes 
in $\RR^n$ and $\Delta = 
\Delta_1 + \cdots + \Delta_p  
\subset \RR^n$ their Minkowski sum. 
Then for a face $\gamma \prec \Delta$ of 
$\Delta$, by taking a point $u \in \RR^n$ 
in the relative interior of its dual cone $\sigma (\gamma)$ 
we define the supporting face 
$\gamma_i \prec \Delta_i$ of $u$ in $\Delta_i$ 
so that we have 
$\gamma = \gamma_1 + \cdots + \gamma_p$. 

\begin{definition}\label{non-deg} 
(see \cite{Oka} etc.) 
Let $g_1, g_2, \ldots , g_p$ be 
Laurent polynomials on $T=(\CC^*)^n$. 
Set $\Delta_i=NP(g_i)$ $(i=1,\ldots, p)$ and 
$\Delta = \Delta_1 + \cdots + \Delta_p$. 
Then we say that the subvariety 
$Z=\{ x\in T=(\CC^*)^n \ | \ g_1(x)=g_2(x)= 
\cdots =g_p(x)=0 \}$ of $T=(\CC^*)^n$ is a 
non-degenerate complete intersection 
if for any face $\gamma \prec \Delta$ of 
$\Delta$ the $p$-form $dg_1^{\gamma_1} \wedge 
dg_2^{\gamma_2} \wedge 
\cdots \wedge dg_p^{\gamma_p}$ does not vanish 
on $\{ x\in T=(\CC^*)^n  \ | \ 
g_1^{\gamma_1}(x)= \cdots =
g_p^{\gamma_p}(x)=0 \}$.
\end{definition}

\begin{definition}\label{def-1}
Let $\Delta_1,\ldots,\Delta_n$ 
be lattice 
polytopes in $\RR^n$. Then 
their normalized $n$-dimensional 
mixed volume 
$\Vol_{\ZZ}( \Delta_1,\ldots,\Delta_n) 
\in \ZZ$ is defined by the formula 
\begin{equation}
\Vol_{\ZZ}( \Delta_1, \ldots , \Delta_n)=
\frac{1}{n!} 
\dsum_{k=1}^n (-1)^{n-k} 
\sum_{\begin{subarray}{c}I\subset 
\{1,\ldots,n\}\\ |I| =k\end{subarray}}
\Vol_{\ZZ}\left(
\dsum_{i\in I} \Delta_i \right)
\end{equation}
where $\Vol_{\ZZ}(\ \cdot\ )
= n! \Vol (\ \cdot\ ) \in \ZZ$ is 
the normalized $n$-dimensional volume 
with respect to the lattice $\ZZ^n 
\subset \RR^n$.
\end{definition}

\begin{theorem}\label{BKK} 
(Bernstein-Khovanskii-Kushnirenko's 
theorem \cite{Khovanskii})\label{thm:14}
Let $g_1, g_2, \ldots , g_p$ be 
Laurent polynomials on $T=(\CC^*)^n$. 
Assume that the subvariety $Z
=\{ x\in T=(\CC^*)^n \ | \ g_1(x)=g_2(x)= 
\cdots =g_p(x)=0 \}$ of $T=(\CC^*)^n$ is a 
non-degenerate complete intersection. 
Set $\Delta_i=NP(g_i)$ $(i=1,\ldots, p)$. Then we have
\begin{equation}
\chi(Z)=(-1)^{n-p} 
\dsum_{\begin{subarray}{c} 
m_1,\ldots,m_p \geq 1\\ m_1+\cdots +m_p=n 
\end{subarray}}\Vol_{\ZZ}(
\underbrace{\Delta_1,\ldots,\Delta_1}_{\text{
$m_1$-times}},\ldots, 
\underbrace{\Delta_p,
\ldots,\Delta_p}_{\text{$m_p$-times}}),
\end{equation}
where $\Vol_{\ZZ}(\underbrace{\Delta_1,
\ldots,\Delta_1}_{\text{$m_1$-times}},
\ldots,\underbrace{\Delta_p,\ldots,
\Delta_p}_{\text{$m_p$-times}})\in \ZZ$ 
is the normalized $n$-dimensional mixed volume 
with respect to the lattice $\ZZ^n 
\subset \RR^n$.
\end{theorem}

Now, let $\Sigma_f$, 
$\Sigma_P$ and $\Sigma_Q$ 
be the dual fans of $\Gamma_{+}(f)$, 
$\Gamma_{+}(P)$ and $\Gamma_{+}(Q)$ 
in $\RR_+^n$ respectively. Then the 
dual fan $\Sigma_f$ of 
the Minkowski sum 
$\Gamma_+(f)= \Gamma_+(P) + \Gamma_+(Q)$ 
is the coarsest common 
subdivision of $\Sigma_P$ and $\Sigma_Q$. 
This implies that for each face 
$\gamma \prec \Gamma_{+}(f)$ 
we have the corresponding faces 
\begin{equation}
\gamma(P) \prec \Gamma_+(P), \qquad 
\gamma(Q) \prec \Gamma_+(Q)
\end{equation}
such that 
\begin{equation}
\gamma = \gamma(P) + \gamma(Q). 
\end{equation}

\begin{definition}\label{def-3}
We say that the meromorphic 
function $f(x)= \frac{P(x)}{Q(x)}$ is 
non-degenerate at the origin $0 \in X= \CC^n$ if for any 
compact face $\gamma$ of $\Gamma_{+}(f)$ the complex 
hypersurfaces $\{x \in T=(\CC^*)^n\ |\ 
P^{\gamma(P)}(x)=0\}$ and $\{x \in T=(\CC^*)^n\ |\ 
Q^{\gamma(Q)}(x)=0\}$ are smooth and reduced 
and intersect transversally in $T=(\CC^*)^n$. 
\end{definition}

For a subset 
$S \subset \{ 1,2, \ldots, n \}$ we set 
\begin{equation}
\RR^S = \{ v=(v_1, v_2, \ldots, v_n) \in \RR^n  \ | \ 
v_i=0 \ (i \notin S) \} \simeq \RR^{|S|} 
\end{equation}
and 
\begin{equation}
\Gamma_+(f)^S= \Gamma_+(f) \cap \RR^S. 
\end{equation}
Similarly, we define 
$\Gamma_+(P)^S, \Gamma_+(Q)^S \subset \RR^S$ 
so that we have 
\begin{equation}
\Gamma_+(f)^S= \Gamma_+(P)^S + \Gamma_+(Q)^S. 
\end{equation}
Let $\gamma_1^S, \gamma_2^S, \ldots, \gamma_{n(S)}^S$ 
be the compact facets of $\Gamma_+(f)^S$ and for 
each $\gamma_i^S$ ($1 \leq i \leq n(S)$) consider 
the corresponding faces 
\begin{equation}
\gamma_i^S(P) \prec \Gamma_+(P)^S, \qquad 
\gamma_i^S(Q) \prec \Gamma_+(Q)^S
\end{equation}
such that 
\begin{equation}
\gamma_i^S= \gamma_i^S(P) + \gamma_i^S(Q). 
\end{equation}
By using the primitive inner conormal vector 
$\alpha_i^S \in \ZZ_+^S \setminus \{ 0 \}$ 
of the facet $\gamma_i^S \prec \Gamma_+(f)^S$ 
we define the lattice distance 
$d_i^S(P)>0$ (resp. $d_i^S(Q)>0$) of 
$\gamma_i^S(P)$ (resp. $\gamma_i^S(Q)$) from 
the origin $0 \in \RR^S$ to be the (unique) 
value of $\alpha_i^S$ on 
$\gamma_i^S(P)$ (resp. $\gamma_i^S(Q)$) and set 
\begin{equation}
d_i^S=d_i^S(P)-d_i^S(Q) \in \ZZ. 
\end{equation}

Finally by using 
the normalized ($|S|-1$)-dimensional volume 
$\Vol_{\ZZ}(\ \cdot\ )$ we set 
\begin{equation}
v_i^S = 
\dsum_{k=0}^{|S|-1} 
\Vol_{\ZZ}(
\underbrace{\gamma_i^S(P),
\ldots,\gamma_i^S(P)}_{\text{
$k$-times}},
\underbrace{\gamma_i^S(Q),
\ldots,\gamma_i^S(Q)}_{\text{($|S|-1-k$)-times}})
\in \ZZ.
\end{equation}
Then we have the following celebrated theorem of 
Gusein-Zade, Luengo and Melle-Hern\'andez 
\cite{G-L-M}. 

\begin{theorem}\label{the-2} 
(Gusein-Zade, Luengo and Melle-Hern\'andez 
\cite{G-L-M}) 
Assume that the meromorphic 
function $f(x)= \frac{P(x)}{Q(x)}$ is 
non-degenerate at the origin $0 \in X= \CC^n$. 
Then we have 
\begin{equation}
\zeta_{f,0}(t) = \prod_{S \not= \emptyset} \Bigl\{ 
\prod_{i: d_i^S >0} 
( 1- t^{d_i^S})^{(-1)^{|S|-1} v_i^S} \Bigr\}. 
\end{equation}
\end{theorem}

Decomposing $X= \CC^n$ into some tori $( \CC^*)^k$ 
as in the proof of \cite[Theorem 3.12]{M-T-1}, 
we can reprove this theorem by Propositions \ref{PDI} 
and \ref{ProI} and Theorem \ref{BKK} 
(see also Varchenko \cite{Varchenko} and Oka \cite{Oka}).  
By Theorems \ref{the-2} and \ref{the-3} 
we obtain the following result. 

\begin{corollary}\label{cor-1} 
Assume that $f(x)= \frac{P(x)}{Q(x)}$ is 
non-degenerate at the origin $0 \in X= \CC^n$ 
and $P(x), Q(x)$ are convenient. 
Then in the notations of Theorem \ref{the-2}, 
for any $\lambda \not=1$ 
the multiplicity of the eigenvalue 
$\lambda$ in $\Phi_{n-1, 0}$ is 
equal to that of the factor $t- \lambda$ 
in the rational function 
\begin{equation}
\prod_{S \not= \emptyset} \Bigl\{ 
\prod_{i: d_i^S >0} 
(t^{d_i^S}-1)^{(-1)^{n-|S|} v_i^S} \Bigr\}
\in \CC (t). 
\end{equation}
\end{corollary}

\section{Mixed Hodge structures of Milnor 
fibers of rational functions}\label{sec:s4}

In this section, by using the meromorphic nearby cycle functors 
introduced in Section \ref{sec:s2} we 
study the mixed Hodge structures of Milnor 
fibers of rational functions and apply 
them to the Jordan normal forms of 
their monodromies. 
Let us consider a rational function $f(x)= \frac{P(x)}{Q(x)}$ 
on $X= \CC^n$ 
such that $0 \in I(f)=P^{-1}(0) \cap Q^{-1}(0)$. 
In order to obtain also a formula for 
the Jordan normal form of its monodromy 
$\Phi_{n-1, 0}$ as in 
Matsui-Takeuchi \cite{M-T-4}, Stapledon \cite{Stapledon} and 
Saito \cite{Saito}, we need the following condition. 

\begin{definition}\label{def-2}
We say that the rational function 
$f(x)= \frac{P(x)}{Q(x)}$ is polynomial-like 
if there exists a 
resolution $\pi_0: \widetilde{X} \rightarrow X= \CC^n$
of singularities of $P^{-1}(0)$ and 
$Q^{-1}(0)$ which induces an isomorphism 
$\widetilde{X} \setminus \pi_0^{-1}( \{ 0 \} ) 
\simto X \setminus \{ 0 \}$ such that for 
any irreducible component $D_i$ of the 
(exceptional) normal crossing divisor 
$D= \pi_0^{-1}( \{ 0 \} )$ we have the condition 
\begin{equation}
{\rm ord}_{D_i} (P \circ \pi_0) > 
{\rm ord}_{D_i} (Q \circ \pi_0). 
\end{equation}
\end{definition}

\begin{proposition}\label{pro-1} 
Assume that the rational function 
$f(x)= \frac{P(x)}{Q(x)}$ is polynomial-like 
and satisfies the 
conditions in Theorem \ref{the-3}. 
Then for any $\lambda \not=1$ the natural 
morphism 
\begin{equation}
\psi_{f, \lambda}^{\merc}
( \CC_X [n] )[-1] \longrightarrow 
\psi_{f, \lambda}^{\mer}
( \CC_X [n] )[-1] 
\end{equation}
is an isomorphism. Morevoer, the supports 
of the both sides of it 
are contained in 
the origin $\{ 0 \} \subset X= \CC^n$. 
\end{proposition}

\begin{proof} 
Let $\pi_0: \widetilde{X} \rightarrow X= \CC^n$ 
be a resolution of singularities of $P^{-1}(0)$ and 
$Q^{-1}(0)$ which induces an isomorphism 
$\widetilde{X} \setminus \pi_0^{-1}( \{ 0 \} ) 
\simto X \setminus \{ 0 \}$ such that for 
any irreducible component $D_i$ of the 
(exceptional) normal crossing divisor 
$D= \pi_0^{-1}( \{ 0 \} )$ we have the condition 
\begin{equation}
{\rm ord}_{D_i} (P \circ \pi_0) > 
{\rm ord}_{D_i} (Q \circ \pi_0). 
\end{equation}
Then by Proposition \ref{ProII} 
we can show that for any $\lambda \not=1$ the natural 
morphism 
\begin{equation}\label{eemor}
\psi_{f \circ \pi_0, \lambda}^{\merc}
( \CC_{\widetilde{X}} [n] )[-1] \longrightarrow 
\psi_{f \circ \pi_0, \lambda}^{\mer}
( \CC_{\widetilde{X}} [n] )[-1] 
\end{equation}
is an isomorphism and the supports 
of the both sides of it 
are contained in $\pi_0^{-1}( \{ 0 \} )$. 
Indeed, let $\widetilde{P^{-1}(0)}$ 
(resp. $\widetilde{Q^{-1}(0)}$) 
be the (smooth) proper transform of 
$P^{-1}(0)$ (resp. $Q^{-1}(0)$) in $\widetilde{X}$. 
Then the divisor 
$\widetilde{P^{-1}(0)} \cup \widetilde{Q^{-1}(0)} 
\cup D$ in $\widetilde{X}$ is normal 
crossing, and the meromorphic function 
$f \circ \pi_0$ on $\widetilde{X}$ takes 
the value $0$ on 
$( \widetilde{P^{-1}(0)} \cup D) 
\setminus  \widetilde{Q^{-1}(0)}$ 
and has a pole of order $1$ along 
$\widetilde{Q^{-1}(0)}$. First, by 
Remark \ref{rem-1} and its analogue for 
meromorphic nearby cycle functor with 
compact support, on 
$( \widetilde{P^{-1}(0)} \cup D) 
\setminus  \widetilde{Q^{-1}(0)}$ 
we have an isomorphism 
\begin{equation}
\psi_{f \circ \pi_0}^{\merc}
( \CC_{\widetilde{X}} [n] )[-1] \simto 
\psi_{f \circ \pi_0}^{\mer}
( \CC_{\widetilde{X}} [n] )[-1].
\end{equation}
Next, with the help of Lemma \ref{lem-1} (ii) 
and its analogue for 
meromorphic nearby cycle functor with 
compact support, by reducing the problem 
to the situation in Proposition \ref{ProII}, 
we can check that for $\lambda \not=1$ the stalk 
of the morphism \eqref{eemor} at each 
point of $( \widetilde{P^{-1}(0)} \cup D) 
\cap  \widetilde{Q^{-1}(0)}$ 
is an isomorphism. 
Moreover, since 
$\widetilde{P^{-1}(0)} \cup \widetilde{Q^{-1}(0)}$ 
is normal crossing in 
$\widetilde{X} \setminus \pi_0^{-1}( \{ 0 \} ) 
\simeq X \setminus \{ 0 \}$, 
by the proof of Theorem \ref{the-3} 
if $\lambda \not=1$ we have 
\begin{equation}
\psi_{f \circ \pi_0, \lambda}^{\merc}
( \CC_{\widetilde{X}} [n] )[-1] \simto  
\psi_{f \circ \pi_0, \lambda}^{\mer}
( \CC_{\widetilde{X}} [n] )[-1] 
\simeq 0
\end{equation}
on $\widetilde{X} \setminus \pi_0^{-1}( \{ 0 \} ) 
\simeq X \setminus \{ 0 \}$. 
Now by Proposition \ref{PDI} 
and its analogue for meromorphic 
nearby cycle functors with compact support, all 
the assertions immediately follow. 
\end{proof} 

Now we shall introduce natural mixed Hodge structures on 
the cohomology groups $H^{j}(F_{0};\CC)_{\lambda}$ 
($\lambda\in \CC$) of the Milnor fiber $F_0$. 
For a variety $Z$ over $\CC$ we 
denote by ${\rm MHM}_{Z}$ the abelian category of 
mixed Hodge modules on $Z$ (see e.g. 
\cite[Section 8.3]{H-T-T} etc.). 
First we regard $\psi_{f,\lambda}^{\mathrm{mero}}
(\CC_{X}[n])[-1]$ (resp. 
$\psi_{f,\lambda}^{\mathrm{mero,c}}(\CC_{X}[n])[-1]$) 
as the underlying perverse sheaf of the mixed 
Hodge module $\psi_{t,\lambda}^{H}({i_{f}}_{*}
(\CC^{H}_{X}[n]|_{X\setminus Q^{-1}(0)}))$ (resp. 
$\psi_{t,\lambda}^{H}({i_{f}}_{!}(
\CC^{H}_{X}[n]|_{X\setminus Q^{-1}(0)}))$) 
$\in {\rm MHM}_{X}$, 
where $\psi_{t,\lambda}^{H}$ (resp. ${i_{f}}_{*}$, 
${i_{f}}_{!}$) is a functor between the categories 
of mixed Hodge modules corresponding to the functor 
$\psi_{t,\lambda}[-1]$ (resp. $R {i_{f}}_{*}$, 
$R{i_{f}}_{!}$) and $\CC_{X}^{H}[n] \in 
{\rm MHM}_{X}$ is the 
mixed Hodge module whose underlying perverse sheaf 
is $\CC_{X}[n]$.
For the inclusion map $j_{0}\colon \{0\}\hookrightarrow X$, 
we consider the pullback 
\begin{equation}
j_{0}^*\psi_{t,\lambda}^{H}
({i_{f}}_{*}(\CC^{H}_{X}[n]|_{X\setminus Q^{-1}(0)})) 
\in \Db ( {\rm MHM}_{\{ 0 \}} ) 
\end{equation}
by $j_{0}$, whose underlying constructible sheaf is 
$j_{0}^{-1}(\psi_{f,\lambda}^{\mathrm{mero}}(\CC_{X}[n])[-1])$. 
Since the $(j-n+1)$-th cohomology group of $j_{0}^{-1}
(\psi_{f,\lambda}^{\mathrm{mero}}(\CC_{X}[n])[-1])$ 
is $H^{j}(F_{0};\CC)_{\lambda}$, we thus obtain a 
natural mixed Hodge structure of $H^{j}(F_{0}; 
\CC)_{\lambda}$. 
In the following, we focus our attention on its weight filtration 
$W_{\bullet}H^{j}(F_{x}; \CC)_{\lambda}$.
Recall that the weight filtration of the (middle 
dimensional) 
cohomology group of the Milnor fiber of an 
isolated hypersurface singular point is the 
monodromy weight filtration (for the definition, 
see e.g. \cite[Section A.2]{M-T-3} etc.) 
of the Milnor monodromy. 
We will show that if the rational function 
$f(x)= \frac{P(x)}{Q(x)}$ 
satisfies the conditions of Proposition \ref{pro-1} 
the cohomology groups of the 
Milnor fiber $F_0$ of $f$ also have a 
similar property. By Proposition \ref{pro-1}, 
we thus obtain the same nice property also 
for the corresponding stalk of 
the perverse sheaf $\psi_{f, \lambda}^{\merc}
( \CC_X [n] )[-1]$ (see Theorem \ref{thm:7-6} 
for its description by a motivic Milnor fiber).  
First, we need the following lemma.

\begin{lemma}\label{weifillem}
Let $g$ be a holomorphic function on a 
complex manifold $Z$.
Moreover, let $M$, $M'$ be mixed Hodge 
modules on $Z$ and 
$M\to M'$ a morphism in the category of 
mixed Hodge modules.
Assume that $M$ (resp. $M'$) has weights 
$\leq l$ (resp. $\geq l$) for some integer 
$l$, i.e. we have $\GR^{W}_{k}M=0\ (k>l)$ (
resp. $\GR^{W}_{k}M'=0\ (k<l)$).
Then, for $\lambda\neq 1$ if the natural 
morphism $\psi^{H}_{g,\lambda}(M)\to 
\psi^{H}_{g,\lambda}(M')$ is an isomorphism,
the weight filtration of $\psi^{H}_{g,\lambda}(M')$ 
is the monodromy weight filtration centered at $l-1$.
\end{lemma}

\begin{proof}
See Appendix \ref{app}. 
\end{proof}

In the situation of Proposition \ref{pro-1},
we set 
\begin{equation}
M:={i_{f}}_{!}(
\CC^{H}_{X}[n]|_{X\setminus Q^{-1}(0)}), 
\quad 
M':={i_{f}}_{*}(
\CC^{H}_{X}[n]|_{X\setminus Q^{-1}(0)}) 
\quad \in \Db ( {\rm MHM}_{X \times \CC} ). 
\end{equation}
Recall that 
$\CC^{H}_{X}[n]$ has a pure weight $n$.
Therefore, by the basic properties of the 
functor ${i_{f}}_{*}$ (resp. ${i_{f}}_{!}$) 
(see \cite[Section 8.3]{H-T-T} etc.) the mixed Hodge module $M$ 
(resp. $M'$) has weights $\leq n$ 
(resp. $\geq n$). 
Then we can apply Lemma~\ref{weifillem} 
to obtain the following theorem. 

\begin{theorem}\label{weifilmonofil}
In the situation of Proposition \ref{pro-1}, for any 
$\lambda\neq 1$ 
the weight filtration of $H^{n-1}(F_{0};
\CC)_{\lambda}$ is the monodromy weight 
filtration of the Milnor monodromy 
$\Phi_{n-1, 0}\colon H^{n-1}(F_{0};
\CC)_{\lambda}\overset{\sim}{\to} H^{n-1}(
F_0 ;\CC)_{\lambda}$ centered at $n-1$.
\end{theorem}

\begin{proof}
Since by Theorem \ref{the-3} we have 
\begin{equation}
H^{j}(F_{0};\CC)_{\lambda}=0 \qquad 
(j\neq n-1), 
\end{equation}
the complex $j_{0}^*\psi_{t,\lambda}^{H}(
{i_{f}}_{*}(\CC^{H}_{X}[n]|_{X\setminus 
Q^{-1}(0)})) \in \Db ( {\rm MHM}_{\{ 0 \}} 
 )$ is quasi-isomorphic to 
$H^0j_{0}^*\psi_{t,\lambda}^{H}({i_{f}}_{*}(
\CC^{H}_{X}[n]|_{X\setminus Q^{-1}(0)}))$ and 
its underlying perverse sheaf is 
$H^{n-1}(F_{0};\CC)_{\lambda}$. 
By Lemma~\ref{weifillem} 
the weight filtration of 
$\psi_{t,\lambda}^{H}({i_{f}}_{*}(
\CC^{H}_{X}[n]|_{X\setminus Q^{-1}(0)}))$
is the monodromy weight filtration centered at $n-1$.
Moreover, in this situation, the support of 
$\psi_{t,\lambda}^{H}({i_{f}}_{*}(
\CC^{H}_{X}[n]|_{X\setminus Q^{-1}(0)}))$ is contained in $\{0\}$.
This implies that 
$\psi_{t,\lambda}^{H}({i_{f}}_{*}(
\CC^{H}_{X}[n]|_{X\setminus Q^{-1}(0)}))$ 
is the zero extension of 
$H^{0}j_{0}^*\psi_{t,\lambda}^{H}(
{i_{f}}_{*}(\CC^{H}_{X}[n]|_{X\setminus Q^{-1}(0)}))$. 
Hence we can identify the weight filtration of 
$H^{0}j_{0}^*\psi_{t,\lambda}^{H}(
{i_{f}}_{*}(\CC^{H}_{X}[n]|_{X
\setminus Q^{-1}(0)}))$
with that of $\psi_{t,\lambda}^{H}(
{i_{f}}_{*}(\CC^{H}_{X}[n]|_{X\setminus Q^{-1}(0)}))$. 
This completes the proof.
\end{proof}

\begin{remark}\label{NJCF} 
For $k\in\ZZ_{> 0}$ and $\lambda \in \CC$ 
we denote by $J_{k,\lambda}$ the number of 
the Jordan blocks in $\Phi_{n-1, 0}$ with size $k$ 
for the eigenvalue $\lambda$. Then 
by Theorem~\ref{weifilmonofil} for $\lambda \not= 1$ 
we can describe $J_{k,\lambda}$ in terms 
of the weight filtration of $H^{n-1}(F_{0};
\CC)_{\lambda}$ as follows:
\begin{equation}\label{formJor}
J_{k,\lambda}=\dim \ \GR^{W}_{n-k}H^{n-1}(F_{0};
\CC)_{\lambda}- \dim \ \GR^{W}_{n-k-2}H^{n-1}(F_{0};\CC)_{\lambda}.
\end{equation}
Moreover the number of 
the Jordan blocks in $\Phi_{n-1, 0}$ with size $\geq k$ 
for the eigenvalue $\lambda$ is equal to 
\begin{equation}\label{formulJor}
\dim \ \GR^{W}_{n-2+k}H^{n-1}(F_{0};
\CC)_{\lambda}+ \dim \ \GR^{W}_{n-1+k}H^{n-1}(F_{0};\CC)_{\lambda}.
\end{equation}
\end{remark}

\section{Motivic Milnor fibers of of rational functions}\label{sect-2}

Following Denef-Loeser \cite{D-L-1},
\cite{D-L-2}, \cite{D-L-3}, Guibert-Loeser-Merle \cite{GLM} 
and Raibaut \cite{Raibaut-1.5}, \cite{Raibaut-2}, we shall define
and study the motivic reincarnations of the Milnor
fibers of rational functions. More precisely, we will define the notion 
of motivic Milnor fiber of rational functions and give 
their relation with the topological properties of the Milnor fiber. 
Such work was studied by Raibaut \cite{Raibaut-2} where the 
author considered elements in the Grothendieck ring under 
the action of $\mathbb{G}_m$. In this section, we use another 
construction for the motivic Milnor fiber, namely, we consider the 
 Grothendieck ring under the action of $\hat{\mu}$ and define the 
notion of motivic Milnor fiber as an element in $\M^{\hat{\mu}}_{\CC}$. In our opinion, in this section we give, 
in some sense, another proof for the results in  \cite{Raibaut-2}.

Let $f(x)= \frac{P(x)}{Q(x)}$ be a rational 
function on $X= \mathbb{C}^n$ and set 
$X_0:=P^{-1}(0)$.  Assume that $0\in P^{-1}(0)
\cap Q^{-1}(0)$. For two positive 
integers $m, r >0$ we set 
\begin{equation}
\mathscr{X}^{r}_{m,1}:=\{\varphi\in 
\mathcal{L}_m(X) \ | \ {\rm ord} Q(\varphi(t))\leq r m, 
f(\varphi(t))= t^m \ \textrm{mod}
\ t^{m+1}\}
\end{equation}
and
\begin{equation}
\mathscr{X}^{r}_{m,1,0}:=\{\varphi\in 
\mathscr{X}^{r}_{m,1} \ | \ \pi^m_0(\varphi)=0 \}
\end{equation}
where $\mathcal{L}_m(X)$ is the space of 
order $m$ arcs on $X$ and $\pi^m_0$ is the truncation 
morphism $\mathcal{L}_m(X)\to X$.
They are locally closed subvarieties of 
$\mathcal{L}_m(X)$ (see, e.g. 
\cite[Section 1]{D-L-3} for the notions of arc spaces). Since $m\geq 1$, 
we see that $\pi^m_0(\mathscr{X}^{r}_{m, 1})
\subset X_0$. Then $\mathscr{X}^{r}_{m,1}$ 
is an $X_0$-variety. Now for 
$m \in \ZZ_{>0}$, let
$\mu_m \simeq \ZZ/\ZZ m$ be the multiplicative
group consisting of the $m$-roots
in $\CC$. We denote by $\hat{\mu}$ the
projective limit $\underset{m}{\varprojlim}
\mu_m$ of the projective system
$\{ \mu_i \}_{i \geq 1}$ with morphisms
$\mu_{im} \longrightarrow \mu_i$
given by $t \longmapsto t^m$. Then 
$\mathscr{X}^{r}_{m,1}$ and 
$\mathscr{X}^{r}_{m,1,0}$ are endowed with 
a good action of $\mu_m$ (and hence of 
$\hat{\mu}$) defined by $\lambda 
\cdot \varphi(t)=\varphi(\lambda t)$.

Following the notations of \cite{D-L-2}, 
for a variety $S$ over $\CC$ we 
denote by $\M^{\hat{\mu}}_{S}$  the ring obtained
from the Grothendieck ring
$\KK_0^{\hat{\mu}}(\Var_{S})$ of
varieties over $S$ with good
$\hat{\mu}$-actions by inverting the
Lefschetz motive $\LL$ which is the class of 
$\mathbb{A}^1_{\CC}\times S \in
\KK_0^{\hat{\mu}}(\Var_{S})$ with trivial 
action of $\hat{\mu}$. Recall also that for an 
$S$-variety $E$ with good $\hat{\mu}$-action, 
we denote its class in  $\KK_0^{\hat{\mu}}
(\Var_{S})$ by $[E, \hat{\mu}]$ (also denoted 
by $[E/S, \hat{\mu}]$, or simply by $[E]$). 
When $S$ consists of only one geometric point, 
i.e. $S={\rm Spec}(\CC)$, we will write 
$\KK_0^{\hat{\mu}}(\Var_{
\CC})$ instead of $\KK_0^{\hat{\mu}}(\Var_{S})$. 
For any $s\in S(\CC)$, there are  natural maps 
$i_s^{-1}:\KK_0^{\hat{\mu}}(\Var_{S})\to 
\KK_0^{\hat{\mu}}(\Var_{\CC})$ and  
$i_s^{-1}:\M^{\hat{\mu}}_{S}\to \M^{\hat{\mu}}_{\CC}$, 
defined by $[E, \hat{\mu}]\mapsto 
[E_s, \hat{\mu}]$ where $E_s$ is the 
fiber at $s$ of $E\to S$. Thus we obtain elements $[
\mathscr{X}^{r}_{m,1}/X_0, \hat{\mu}]$ (or $[
\mathscr{X}^{r}_{m,1}, \hat{\mu}],$ or $ [
\mathscr{X}^{r}_{m,1}]$) of $\M^{\hat{\mu}}_{X_0}$ and $ [
\mathscr{X}^{r}_{m,1,0}]$ of 
$\M^{\hat{\mu}}_{\CC}$. Note that $ [
\mathscr{X}^{r}_{m,1,0}]=i_0^{-1}( [
\mathscr{X}^{r}_{m,1}])$.

\begin{definition}
For a positive integer $r >0$ and the rational function 
$f=\frac{P}{Q}: 
X\to \mathbb{A}^1_{\mathbb{C}}$ we define a power 
series $Z^{r}_f(T)$ of $T$ over $\M^{\hat{\mu}}_{X_0}$ by
\begin{equation}
Z^{r}_f(T)= \sum_{m \geq 1} \Bigl( [
\mathscr{X}^{r}_{m,1}/X_0, \hat{\mu}] \cdot \LL^{-mn}
\Bigr) T^m.
\end{equation}
We call it a motivic zeta function of $f$. 
\end{definition}

We recall the notion of rational series. 

\begin{definition}
Let $\mathscr A$ be one 
of the following rings
\begin{equation}
\mathbb{Z}[\LL,\LL^{-1}], \quad \mathbb{Z}
\left[\LL,\LL^{-1},\frac{1}{1-\LL^{-i}}\right]_{i>0}, 
\quad\text{and}\ \ \M_S^{\hat\mu}.
\end{equation}
Let $\mathscr A[[T]]_{\textrm {sr}}$ be the 
$\mathscr A$-submodule of $\mathscr A[[T]]$ 
generated by $1$ and by finite products of 
elements of the form $\frac{\LL^aT^b}{1-\LL^aT^b}$ 
with $a$ in  $\mathbb{Z}$ and 
$b$ in $\ZZ_{>0}$. 
\end{definition}

By \cite{D-L-1}, there 
is a unique $\mathscr A$-linear homomorphism 
\begin{equation}
\lim_{T\rightarrow\infty}: \mathscr 
A[[T]]_{\textrm {sr}}\rightarrow \mathscr A
\end{equation}
such that 
\begin{equation}
\lim_{T\rightarrow\infty}\frac{\LL^aT^b}{1-\LL^aT^b}=-1.
\end{equation}
For a finite set $I$, we consider rational 
polyhedral convex cones in $\mathbb{R}_{> 0}^I$. 
By this, we mean a convex subset of 
$\mathbb{R}_{> 0}^I$ defined by a finite 
number of integral linear inequalities of type 
$l\geq 0$ or $l>0$ and stable by 
multiplication by $\mathbb{R}_{>0}$. Let $\Delta$ 
be a rational polyhedral convex cone in 
$\mathbb{R}_{>0}^I$ and let $\overline{\Delta}$ 
denote its closure in $\mathbb{R}_{\geq 0}^I$. 
Let $l$ and $v$ be two integral linear forms 
on $\mathbb{Z}^I$ positive on 
$\overline{\Delta}\setminus\{(0,\dots,0)\}$. 
Let us consider the series
\begin{equation}
S_{\Delta,l,v}(T):=\sum_{a\in\Delta\cap
\mathbb{N}_{>0}^I}\LL^{-v(a)}T^{l(a)}
\end{equation}
in $\mathbb{Z}[\LL,\LL^{-1}][[T]]$. 
Then we have the following lemma (see 
\cite[Lemma 2.1.5]{Guibert}  and  \cite[Section 2.9]{GLM}).

\begin{lemma}\label{lmrational}
With the above notations, the series 
$S_{\Delta,l,v}(T)$ lies in $\mathbb{Z}[
\LL,\LL^{-1}][[T]]_{\textrm {sr}}$ and the limit 
$\lim_{T\to \infty}S_{\Delta,l,v}(T)$ is 
equal to $\chi_c(\Delta)$ i.e. the Euler characteristic 
with compact support of $\Delta$. 
In particular, if $\Delta$ is a rational polyhedral 
convex cone in $\mathbb{R}_{> 0}^I$ defined by
\begin{equation}
\sum_{i\in K}a_ix_i\leq \sum_{i\in I\setminus K}a_ix_i
\end{equation}
with $a_i\in \ZZ_+$, $a_i>0$ for $i\in K$, 
$K$ and $I\setminus K$ non-empty, then 
we have $\lim_{T\to \infty}S_{\Delta,l,v}(T)=0.$
\end{lemma}

We shall describe the motivic zeta function 
of $f$ in terms of log-resolutions. For this purpose, 
assume that $f$ satisfies the 
conditions in Theorem \ref{the-3}. Then there exists a 
resolution of singularities 
$\pi_0: \widetilde{X} \rightarrow X= \CC^n$
of $P^{-1}(0)$ and 
$Q^{-1}(0)$ which induces an isomorphism
$\widetilde{X} \setminus \pi_0^{-1}( \{ 0 \} )
\simto X \setminus \{ 0 \}$ such that $\pi_0^{-1}( \{ 0 \} )$ 
and $\pi_0^{-1}(P^{-1}(0)\cup Q^{-1}(0))$ are 
strict normal crossing divisors in 
$\widetilde{X}$. Denote by $D_i$ ($1 \leq i \leq m$), 
the irreducible components of the normal 
crossing divisor $D= \pi_0^{-1}( \{ 0 \} )$, 
let $D_P= \widetilde{P^{-1}(0)}$ and 
$D_Q=~\widetilde{Q^{-1}(0)}$ be the (smooth) proper
transforms of $P^{-1}(0)$ and $Q^{-1}(0)$ in 
$\widetilde{X}$ respectively. Note 
that they intersect transversally. Then the 
rational function $f \circ \pi_0$ on
$\widetilde{X}$ has some points of indeterminacy. As in the proof 
of \cite[Theorem 3.6]{M-T-2} we construct towers of 
blow-ups of $\widetilde{X}$ over
it to eliminate the points of indeterminacy of 
$f \circ \pi_0$. Then we obtain a proper 
morphism $\pi_1 : Y \longrightarrow \widetilde{X}$ 
of smooth complex varieties.
Set $\pi = \pi_0 \circ \pi_1 : Y \longrightarrow
X$ and let $\pi^{-1}(0)= \cup_{i=1}^k E_i$ be the 
irreducible decomposition of the normal crossing 
divisor $\pi^{-1}(0)$ in $Y$. Let $E_P$ and 
$E_Q$ be the (smooth) proper
transforms of $D_P= \widetilde{P^{-1}(0)}$ and 
$D_Q= \widetilde{Q^{-1}(0)}$ in $Y$ respectively. 
By our construction of $Y$, the divisor 
$\pi^{-1}(0) \cup E_P \cup E_Q$
in $Y$ is strict normal crossing. Denote by 
$G$ the union of its irreducible components 
along which the order of the rational function 
$g=f \circ \pi$ is $\leq 0$ so that we have $E_Q \subset G$.
Now we define an open subset $\Omega$ of $Y$ 
by $\Omega =Y \setminus G$ and set 
\begin{equation}
U= \pi^{-1}(0) \setminus G= \pi^{-1}(0)
\cap \Omega \subset \pi^{-1}(0).
\end{equation}
Then we have an isomorphism
\begin{equation}\label{eqqq}
\psi_{f}^{\merc}( \CC_X )_0 \simeq
R \Gamma_c (U; \psi_{f \circ \pi}( \CC_Y )).
\end{equation}
For each $i\in \{1, \ldots, k\}\cup\{P\}$, let $N_i(P), N_i(Q)>0$ 
be the orders of the zeros of $P \circ \pi, 
Q \circ \pi$ along $E_i$ and $b_i:=N_i(P)- N_i(Q) \in \ZZ$ 
that of $g=f\circ \pi$ along $E_i$. Set 
\begin{equation}
C:=\{i\in \{1, \ldots, k\}: b_i >0 \}\cup\{P\}.
\end{equation}
For each $i\in \{1, \ldots, k\}\cup\{P\}$, 
we also denote by $\nu_i-1$ the multiplicity of $E_i$ 
in the divisor of $\pi^{*}dx$, where $dx$ is a local 
non-vanishing volume form at $0$, i.e. a local
generator of the sheaf of differential forms of 
maximal degree at $0$. Note that we have $\nu_i>0$. 

 For a non-empty subset $I \subset \{1,2,\ldots, k, P\}$, 
set $E_I= \bigcap_{i \in I} E_i$, 
\begin{equation}
E_I^{\circ}=E_I \setminus \left\{ \(
\bigcup_{i \notin I}E_i\) \cup
E_Q \right\} \subset Y
\end{equation}
and $d_I= {\rm gcd} (b_i)_{i \in I}>0$. Then, 
as in \cite[Section 2.3]{D-L-3}, we can construct 
an unramified Galois covering $\tl{E_I^{\circ}} 
\longrightarrow E_I^{\circ}$
of $E_I^{\circ}$ as follows. First, for a point 
$p \in E_I^{\circ}$ we take an Zariski affine 
open neighborhood $W$ of $p$ in $Y$ on which there 
exist regular functions
$\xi_i$ $(i \in I)$ such that $E_i \cap 
W=\{ \xi_i=0 \}$ for any $i \in I$. Then on $W$ we 
have $g=g_{1,W} (g_{2,W})^{d_I}$,
where we set $g_{1,W}=g \prod_{i \in I}\xi_i^{-b_i}$ 
and $g_{2,W}=\prod_{i \in I} \xi_i^{\frac{b_i}{d_I}}$. 
Note that $g_{1,W}$ is a unit on $W$ and $g_{2,W} 
\colon W \longrightarrow \CC$ is a regular 
function. It is easy to see that $E_I^{\circ}$ 
is covered by such affine open subsets $W$. 
Then as in \cite[Section 2.3]{D-L-3} by gluing the varieties 
\begin{equation}\label{eq:6-26}
\tl{E_{I,W}^{\circ}}:=\{(t,z) \in
\CC^* \times (E_I^{\circ}
 \cap W) \ |\
t^{d_I} =(g_{1,W})^{-1}(z)\}
\end{equation}
together in an obvious way, we obtain the 
variety $\tl{E_I^{\circ}}$ over $E_I^{\circ}$ 
as an $X_0$-variety.  The unramified Galois covering
$\tl{E_I^{\circ}}$ of $E_I^{\circ}$ admits 
a natural $\mu_{d_I}$-action defined by assigning the automorphism
$(t,z) \longmapsto (\zeta_{d_I} t, z)$
of $\tl{E_I^{\circ}}$ to the generator
$\zeta_{d_I}:=\exp
(2\pi\sqrt{-1}/d_I) \in \mu_{d_I}$. Namely
the variety $\tl{E_I^{\circ}}$ is
equipped with a good $\hat{\mu}$-action
in the sense of \cite[Section
2.4]{D-L-2}. Then $[\tl{E_I^{\circ}}, 
\hat{\mu}]$ is an element in $\KK_0^{\hat{\mu}}(\Var_{X_0})$.

By the same argument as in proof of 
\cite[Theorem 2.4]{D-L-3} we obtain the following.

\begin{lemma} \label{contactloci}
With the previous notations, for any 
positive integers 
$r, m>0$, we have the following equality in $\M^{\hat{\mu}}_{X_0}$
\begin{equation}
[\mathscr{X}^{r}_{m,1}]= \LL^{mn}
\sum_{\begin{subarray}{c}I\subset
\{1,\ldots,k\}\cup\{P\}\\ I\neq 
\emptyset\end{subarray}}(\LL-1)^{|I|-1}[
\tl{E_I^{\circ}}]\left(\sum_{\begin{subarray}{c}k_i\geq 1,
 i\in I\\ \sum k_ib_i=m, \sum k_iN_i(Q)\leq 
r \sum k_ib_i\end{subarray}} \LL^{
-\sum_{i\in I}k_i\nu_i}\right).
\end{equation}
\end{lemma}

The Euler characteristic of complex 
constructible sets can be regarded as a ring homomorphism 
\begin{equation}
\chi: \KK_0^{\hat{\mu}}(\Var_{\CC}) \longrightarrow \mathbb{Z}.
\end{equation}
Since we have $\chi(\LL)=1$, it extends uniquely 
to a ring homomorphism  
\begin{equation}
\chi:  \M^{\hat{\mu}}_{\CC} \longrightarrow \mathbb{Z}.
\end{equation}
The following result is an analogue for rational 
function of \cite[Theorem 1.1]{D-L-3} and is also obtained in \cite{Raibaut-2} under a different construction.

\begin{corollary}\label{Cor-MLF}
There exists $r_0 \gg 0$ such that for any 
$r > r_0$ and $m \geq 1$, 
the Lefschetz number $\Lambda (m)_{f,0}$ of 
$f$ at the origin $0\in X$ is equal to 
$\chi(\mathscr{X}^{r}_{m,1,0})$.
\end{corollary}

\begin{proof}
With the previous notations, we have 
$\chi((\LL-1)^{|I|-1})=0$ for any $I$ with $|I|>1$. 
Then, it follows from Lemma \ref{contactloci} that
\begin{equation}
\chi(\mathscr{X}^{r}_{m,1,0})=
\sum_{\begin{subarray}{c}i\in\{1, \ldots, 
k\}, b_i>0\\b_i|m, N_i(Q)\leq r b_i
\end{subarray}}b_i\chi(E^{\circ}_i).
\end{equation}
Therefore, if $r> r_0:= 
\sup_{\{i : b_i>0\}}\frac{N_i(Q)}{b_i}$, we get
\begin{equation}
\chi(\mathscr{X}^{r}_{m,1,0}
)=\sum_{\begin{subarray}{c}i\in\{1, \ldots, k\}, 
b_i>0\\b_i|m\end{subarray}}b_i\chi(E^{\circ}_i).
\end{equation} 
Combining this with Corollary \ref{cor-111} we obtain 
\begin{equation}
\Lambda (m)_{f,0}=\chi(\mathscr{X}^{r}_{m,1,0}).
\end{equation}
\end{proof}

The following results and definitions are 
inspired by \cite[Section 3.5]{D-L-2}, 
\cite[Section 2]{D-L-3}, \cite[Section 3.8]{GLM}, \cite[Section 1.4]{Raibaut-1.5} and \cite[Section 4.1]{Raibaut-2}. The result below could be implied from \cite[Theorem 6]{Raibaut-2}, though we provide here a different proof.

\begin{theorem}\label{rationalthm} 
There exists $r_0 \gg 0$ such that for any 
$r > r_0$ the series $Z^{r}_f(T)$ is 
rational, and it is independent of 
$r > r_0$. Moreover for $r > r_0$ 
the limit 
$\lim_{T\to\infty}Z^{r}_f(T) \in \M^{\hat{\mu}}_{X_0}$ 
exists. For $r > r_0$ we set 
\begin{equation}
\mathcal{S}^{\merc}_f:= -\lim_{T\to \infty}
Z^{r}_f(T) \in \M^{\hat{\mu}}_{X_0}.
\end{equation}
Then we have 
\begin{equation}
\mathcal{S}^{\merc}_f=\sum_{I\subset C, I\neq \emptyset}
(1-\LL)^{|I|-1}[\tl{E_I^{\circ}}, 
\hat{\mu}].
\end{equation}
\end{theorem}

\begin{proof}
It follows from Lemma \ref{contactloci} that
\begin{equation}
Z^{r}_f(T)=\sum_{\begin{subarray}{c}I\subset
\{1,\ldots,k\}\cup\{P\}\\ I\neq \emptyset
\end{subarray}}(\LL-1)^{|I|-1}[\tl{E_I^{\circ}}] 
\sum_{m \geq 1} \ 
\left(\sum_{\begin{subarray}{c}k_i\geq 1, i\in I, 
m\geq 1\\ \sum k_ib_i=m, \sum k_iN_i(Q)\leq 
r \sum k_ib_i\end{subarray}} \LL^{-\sum_{i\in I}k_i\nu_i}T^m\right).
\end{equation}
For each nonempty subset $I$ of the set  
$\{1,\ldots,k\}\cup\{P\}$, we consider the cone 
\begin{equation}
\Delta^{r}_I:=\left\{(m, (k_i))\in\mathbb{R}_{>0}
\times\mathbb{R}_{>0}^I \ | \ \sum_{i\in I}
 k_iN_i(Q)\leq r \sum_{i\in I} k_ib_i, \ 
\sum_{i\in I} k_ib_i=m\right\}.
\end{equation}
Let us consider also the linear forms $l, v$ on 
$\mathbb{R}_{>0}\times\mathbb{R}_{>0}^I$ defined by 
\begin{equation}
l(m, (k_i))=m, \qquad v(m, (k_i))=\sum_{i\in I}\nu_ik_i.
\end{equation}
Then we can rewrite the motivic zeta function as follows 
\begin{equation}
Z^{r}_f(T)=\sum_{\begin{subarray}{c}I\subset
\{1,\ldots,k\}\cup\{P\}\\ I\neq \emptyset
\end{subarray}}(\LL-1)^{|I|-1}[\tl{E_I^{\circ}}]
\cdot S_{\Delta^{r}_I, l, v}(T).
\end{equation}
It is easy to see that $\Delta^{r}_I$ is a 
rational polyhedral cone and the integral 
linear forms $l, v$ are positive on 
$\overline{\Delta^{r}_I}\setminus \{0\}$. 
Then by Lemma \ref{lmrational} the series 
$Z^{r}_f(T)$ is rational.

First, consider the case 
where $I\subset C$. Then we have $b_i>0$ for any 
$i\in I$. So for $r>r_0:= \sup_{i\in I}N_i(Q)/b_i$ we get
\begin{equation}
S_{\Delta^{r}_I, l, v}(T)=
\prod_{i\in I} \frac{\LL^{-\nu_i}T^{b_i}}{1-\LL^{-\nu_i}T^{b_i}}
\end{equation}
and hence $\lim_{T\to \infty}S_{\Delta^{r}_I, l, v}(T)
=(-1)^{|I|}$. Next, consider the case where 
$K:=I\setminus C\neq\emptyset$. 
Then we have $\lim_{T\to \infty}S_{\Delta^{r}_I, l, v}
(T)=\chi_c(\Delta^{r}_I)$. Nevertheless,  
the cone $\Delta^{r}_I$ is homeomorphic to 
the following one in $\mathbb{R}_{>0}^I$: 
\begin{equation}
\left\{ (k_i)\in\mathbb{R}_{>0}^I \ | \ 
\sum_{i\in K} k_i(N_i(Q)- r b_i)\leq
\sum_{i\in I\setminus K} k_i(r b_i-N_i(Q))\right\}.
\end{equation}
By Lemma \ref{lmrational} 
its Euler characteristic with compact 
support is equal to $0$. This completes the proof. 
\end{proof}

Applying the base change morphism $\textrm{Fiber}_0: 
\M^{\hat{\mu}}_{X_0}\to \M^{\hat{\mu}}_{\CC}$,
defined by $[A/X_0, \hat{\mu}]\mapsto 
[A\times_{X_0}0, \hat{\mu}]$ we set 
\begin{equation}
\mathcal{S}_{f, 0}^{\merc}:= 
\textrm{Fiber}_0(\mathcal{S}^{\merc}_f) 
\in \M^{\hat{\mu}}_{\CC}.
\end{equation}

\begin{definition}
We call $\mathcal{S}_{f,0}^{\merc} \in \M^{\hat{\mu}}_{\CC}$ 
the motivic Milnor fiber with compact support of $f$ 
at the origin $0 \in X= \CC^n$. 
\end{definition}

By Theorem \ref{rationalthm} we obtain the following result. The formula in this result could be obtained from \cite[Proposition 4]{Raibaut-2}.  

\begin{theorem}
The motivic Milnor fiber with compact support
$\mathcal{S}_{f,0}^{\merc}$ of 
$f$ at the origin $0 \in X= \CC^n$ is written as 
\begin{equation}\label{MMF}
\mathcal{S}_{f,0}^{\merc}
=\sum_{\begin{subarray}{c}I\subset \{1, \ldots, k\}
\cap C \\ I \neq \emptyset\end{subarray}}
\Big\{
(1-\LL)^{|I| -1}
[\tl{E_I^{\circ} \setminus E_P}]
+
(1-\LL)^{|I|}
[E_I^{\circ} \cap E_P]
\Big\} \in
\M_{\CC}^{\hat{\mu}},
\end{equation}
where $[E_I^{\circ} \cap E_P] \in
\M_{\CC}^{\hat{\mu}}$ is endowed
with the trivial action of
$\hat{\mu}$.
\end{theorem}

As in \cite[Section 3.1.2 and 3.1.3]{D-L-2},
we denote by $\HSm$ the abelian
category of Hodge structures with a
quasi-unipotent endomorphism. Then, to
the object $\psi_{f}^{\merc}( \CC_X )_0 \in
\Dbc(\{ 0 \})$ and the
semisimple part of the monodromy
automorphism acting on it, we can associate
an element
\begin{equation}
[H_{f,0}^{\merc}] = \sum_{j \in \ZZ} (-1)^j
[H^j \psi_{f}^{\merc}( \CC_X )_0  ]
\in \KK_0(\HSm)
\end{equation}
as in \cite{D-L-1} and \cite{D-L-2},
where the weight filtration of
the limit mixed Hodge structure
$[H^j \psi_{f}^{\merc}( \CC_X )_0] \in \HSm$
is the ``relative" monodromy filtration
defined by the Milnor monodromy
of $\psi_{f}^{\merc}( \CC_X )_0$.
To describe the element
$[H_{f,0}^{\merc}]\in \KK_0(\HSm)$ in terms of
$\mathcal{S}_{f,0}^{\merc}\in \M_{\CC}^{\hat{\mu}}$, let
\begin{equation}
\chi_h \colon \M_{\CC}^{\hat{\mu}}
\longrightarrow \KK_0(\HSm)
\end{equation}
be the Hodge characteristic morphism
defined in \cite{D-L-2} which
associates to a variety $Z$ with a
good $\mu_d$-action the Hodge structure
\begin{equation}
\chi_h ([Z])=\sum_{j \in \ZZ} (-1)^j
[H_c^j(Z;\QQ)] \in \KK_0(\HSm)
\end{equation}
with the actions induced by the one
$z \longmapsto \exp (2\pi\sqrt{-1}/d)z$
($z\in Z$) on $Z$. Then
as in \cite[Theorem 4.4]{M-T-4} and
\cite{Raibaut-1}, by applying
\cite[Theorem 4.2.1]{D-L-1} and
\cite[Section 3.16]{GLM}
to our situation \eqref{eqqq},
we obtain the following
result. This result could be also implied from \cite[Theorem 8]{Raibaut-2}.

\begin{theorem}\label{thm:7-6}
In the Grothendieck group $\KK_0(\HSm)$, we have
the equality
\begin{equation}
[H_{f,0}^{\merc}]=\chi_h( \mathcal{S}_{f,0}^{\merc}).
\end{equation}
\end{theorem}

From now on, we assume that $f$ is polynomial-like 
and satisfies the 
conditions in Theorem \ref{the-3}. 
For an element $[V] \in \KK_0(\HSm)$,
$V \in \HSm $ with a
quasi-unipotent endomorphism
$\Theta \colon V \simto V$, $p, q \geq 0$ and
$\lambda \in \CC$ denote by
$e^{p,q}([V])_{\lambda}$ the dimension of the
$\lambda$-eigenspace of the morphism
$V^{p,q} \simto V^{p,q}$ induced by
$\Theta$ on the $(p,q)$-part $V^{p,q}$ of $V$.
Then by Proposition \ref{pro-1} and 
Theorem \ref{weifilmonofil}
we obtain the following result.

\begin{corollary}\label{ICR}
Assume that $\lambda \not= 1$.
Then we have $e^{p,q}(
[H_{f,0}^{\merc}])_{\lambda}=0$
for $(p,q) \notin [0,n-1] \times [0,n-1]$.
Moreover for any $(p,q) \in [0,n-1]
\times [0,n-1]$ we have
the Hodge symmetry
\begin{equation}
e^{p,q}( [H_{f,0}^{\merc}])_{\lambda}=e^{n-1-q,n-1-p}(
[H_{f,0}^{\merc}])_{\lambda}.
\end{equation}
\end{corollary}

\begin{definition}\label{def-HSP}
We define a Puiseux series
$\tl{\rm sp}_{f,0}(t)$ with coefficients in $\ZZ$ by
\begin{equation}
\tl{\rm sp}_{f,0}(t)
= \sum_{\alpha \in (\QQ \setminus \ZZ) \cap (0,n)}\left(
 \dim \ \GR^{\lfloor{\alpha}\rfloor}_{F}
H^{n-1}(F_{0};\CC)_{\exp(2\pi\sqrt{-1}\alpha)} \right)
t^{\alpha}.
\end{equation}
We call it the reduced Hodge spectrum of $f$ at the
origin $0 \in \CC^n$.
\end{definition}
By Corollary \ref{ICR} and
\begin{equation}
e^{p,q}( [H_{f,0}^{\merc}])_{\lambda}=
e^{q,p}( [H_{f,0}^{\merc}])_{\overline{\lambda}}
\end{equation}
we obtain the symmetry of the
reduced Hodge spectrum 
\begin{equation}
\tl{\rm sp}_{f,0}(t) = t^n \cdot
\tl{\rm sp}_{f,0} \Bigl( \frac{1}{t} \Bigr)
\end{equation}
centered at $\frac{n}{2}$.

We can reduce $\mathcal{S}_{f,0}^{\merc} \in
\M_{\CC}^{\hat{\mu}}$ as follows. 
Let $\pi_0: \widetilde{X} \rightarrow X= \CC^n$ 
be a resolution of singularities of $P^{-1}(0)$ and 
$Q^{-1}(0)$ which induces an isomorphism 
$\widetilde{X} \setminus \pi_0^{-1}( \{ 0 \} ) 
\simto X \setminus \{ 0 \}$ such that for 
any irreducible component $D_i$ ($1 \leq i \leq m$) of the 
(exceptional) normal crossing divisor 
$D= \pi_0^{-1}( \{ 0 \} )$ we have the condition 
\begin{equation}
{\rm ord}_{D_i} (P \circ \pi_0) > 
{\rm ord}_{D_i} (Q \circ \pi_0). 
\end{equation}
For $1 \leq i \leq m$ 
let $a_i>0$ be 
the order of $\tl{g}=f \circ \pi_0$
along $D_i$. Namely we set 
\begin{equation}
a_i= {\rm ord}_{D_i} (P \circ \pi_0)- 
{\rm ord}_{D_i} (Q \circ \pi_0) >0. 
\end{equation} 
Let $D_P= \widetilde{P^{-1}(0)}$ and 
$D_Q=~\widetilde{Q^{-1}(0)}$ be the (smooth) proper
transforms of $P^{-1}(0)$ and $Q^{-1}(0)$ in 
$\widetilde{X}$ respectively. 
For a non-empty subset 
$I \subset \{1,2,\ldots, m\}$, set 
$D_I= \bigcap_{i \in I} D_i$,
\begin{equation}
D_I^{\circ}=D_I \setminus \left\{ \(
\bigcup_{i \notin I}D_i\) \cup
D_Q
 \right\} \subset \widetilde{X}
\end{equation}
and $e_I = {\rm gcd} (a_i)_{i \in I} >0$.
Then we can construct an unramified
Galois covering $\tl{D_I^{\circ}
\setminus D_P}$
of $D_I^{\circ} \setminus D_P$ with a natural
$\mu_{e_I}$-action as above.
Let $[\tl{D_I^{\circ}
\setminus D_P }]$ be the element
of the ring $\M_{\CC}^{\hat{\mu}}$
which corresponds to $\tl{D_I^{\circ}
\setminus D_P}$.
Then as in the proof of \cite[Theorem 4.7]{M-T-3} 
we obtain the 
following result. Define an element
$\R_{f,0}^{\merc} \in
\M_{\CC}^{\hat{\mu}}$ by
\begin{equation}\label{MMFT}
\R_{f,0}^{\merc}
=\sum_{I \neq \emptyset}
\Big\{
(1-\LL)^{|I| -1}
[\tl{D_I^{\circ} \setminus D_P}]
+
(1-\LL)^{|I|}
[D_I^{\circ} \cap D_P]
\Big\} \in
\M_{\CC}^{\hat{\mu}},
\end{equation}
where $[D_I^{\circ} \cap D_P] \in
\M_{\CC}^{\hat{\mu}}$ is endowed
with the trivial action of
$\hat{\mu}$.

\begin{theorem}\label{thm:7-7}
In the Grothendieck group $\KK_0(\HSm)$, we have
the equality
\begin{equation}
\chi_h( \mathcal{S}_{f,0}^{\merc}) =
\chi_h( \R_{f,0}^{\merc}).
\end{equation}
\end{theorem}

From now on, we shall
rewrite our formula for $\R_{f,0}^{\merc} \in
\M_{\CC}^{\hat{\mu}}$
more explicitly by using the Newton polyhedron
$\Gamma_+(f)$ of $f$. For this purpose,
we assume that the rational function
$f(x)= \frac{P(x)}{Q(x)}$ is
non-degenerate at the origin $0 \in X= \CC^n$
and $P(x), Q(x)$ are convenient. For $f$ to be
polynomial-like, we assume moreover that
$\Gamma_+(P)$ is properly contained in
$\Gamma_+(Q)$ in the following sense.

\begin{definition}\label{def-29}
We say that the Newton polyhedron
$\Gamma_+(P)$ is properly contained in the one
$\Gamma_+(Q)$ if for any vector
$u \in \Int ( \RR^n_+)$ in the interior
$\Int ( \RR^n_+)$ of $\RR^n_+$ we have
\begin{equation}
\min_{v \in \Gamma_+(P)} \langle u, v \rangle
>
\min_{v \in \Gamma_+(Q)} \langle u, v \rangle.
\end{equation}
In this case, we write $\Gamma_+(P) \subset
\subset \Gamma_+(Q)$.
\end{definition}

We use the notations
in Section \ref{sec:s3}. For a compact face
$\gamma \prec \Gamma_{+}(f)$ of $\Gamma_{+}(f)$ let
\begin{equation}
\gamma(P) \prec \Gamma_+(P), \qquad
\gamma(Q) \prec \Gamma_+(Q)
\end{equation}
be the corresponding faces such that
\begin{equation}
\gamma = \gamma(P) + \gamma(Q).
\end{equation}
Let $\square_{\gamma} \subset \RR^n$ be
the convex hull of $\gamma (P)$ and
$\gamma (Q)$. We define the Cayley
polyhedron $\Gamma_+(P) * \Gamma_+(Q)
\subset \RR^{n+1}$ to be the convex hull of
\begin{equation}
( \Gamma_+(P) \times \{ 0 \} ) \cup
( \Gamma_+(Q) \times \{ 1 \} )
\end{equation}
in $\RR^{n+1}$.

\begin{example}
Let $P(x,y)=x^4+y^4, Q(x, y)=x^2+xy+y^3.$ 
Then, one can check that the rational 
function $f(x, y)=\frac{P(x, y)}{Q(x, y)}$ is 
non-degenerate at the origin $0 \in X= \CC^2$ 
and $\Gamma_+(P) \subset
\subset \Gamma_+(Q)$. The Newton polyhedron of 
$P$ has one compact facet $AB$ and that of $Q$ has 
two compact ones $CD$ and $DE$ (see Figure 1). 
For the face $\gamma= FG$ of $\Gamma_{+}(f)$ we have 
$$\gamma(P)= AB, \qquad \gamma(Q)= CD,$$ 
and $\square_{\gamma}$ is the convex hull 
of the four points $A, B, C, D$.
\end{example}

By our assumption
$\Gamma_+(P) \subset
\subset \Gamma_+(Q)$, the first projection
$\RR^{n+1} = \RR^{n} \times \RR^{1}
\rightarrow \RR^{n}$ induces an isomorphism
of a side face $\tilde{\gamma}$ of
$\Gamma_+(P) * \Gamma_+(Q)$ to $\square_{\gamma}$.
Hence we have $\dim \square_{\gamma} = \dim \gamma
+1$. Denote by
$\LL( \square_{\gamma} )
\simeq \RR^{\dim \gamma+1}$
the linear subspace of $\RR^n$ parallel to
the affine span ${\rm Aff}( \square_{\gamma} )
\simeq \RR^{\dim \gamma+1}$
of $\square_{\gamma}$ in $\RR^n$.
Let $H( \gamma, P)$ (resp. $H( \gamma, Q)$) $\subset
{\rm Aff}( \square_{\gamma} )$ be the
affine hyperplane of ${\rm Aff}( \square_{\gamma} )$
containing $\gamma (P)$ (resp. $\gamma (Q)$)
and parallel to ${\rm Aff}( \gamma )$.
By a suitable choice of a
translation isomorphism
${\rm Aff}( \square_{\gamma} )
\simeq \LL( \square_{\gamma} )$,
we may assume that the image of $H( \gamma, Q)
\subset {\rm Aff}( \square_{\gamma} )$ in
$\LL( \square_{\gamma} )$ passes through
the origin $0 \in \LL( \square_{\gamma} )$.
Denote by $L( \gamma, P)$ (resp. $L( \gamma, Q)$) $\subset
\LL( \square_{\gamma} )$ the image of
$H( \gamma, P)$ (resp. $H( \gamma, Q)$).
Let $M_{\gamma}=\ZZ^n
\cap \LL(\square_{\gamma})
\simeq \ZZ^{\dim \gamma+1}$ be the lattice in
$\LL(\square_{\gamma}) \simeq
\RR^{\dim \gamma+1}$.
In the dual
$\LL(\square_{\gamma})^* \simeq
\RR^{\dim \gamma+1}$ of $\LL(\square_{\gamma})$
consider also its dual lattice $M_{\gamma}^*
\simeq \ZZ^{\dim \gamma+1}$.
We define a one dimensional subspace
$L( \gamma, Q)^{\perp} \simeq \RR$
of $\LL(\square_{\gamma})^*$ by
\begin{equation}
L( \gamma, Q)^{\perp}= \{
u \in \LL(\square_{\gamma})^* \ | \
\langle u,v \rangle =0 \
(v \in L( \gamma, Q)) \}
\subset \LL(\square_{\gamma})^*.
\end{equation}
Let $\alpha_{\gamma} \in
( L( \gamma, Q)^{\perp} \cap
M_{\gamma}^*) \setminus \{ 0 \}
\simeq \ZZ \setminus \{ 0 \}$
be the primitive vector
whose value
on $L( \gamma, P) \subset \LL(\square_{\gamma})$
is a positive integer.
We call it the lattice distance of $L( \gamma, P)$
from $L( \gamma, Q)$ and denote it by $d_{\gamma} >0$.
By using the lattice $M_{\gamma}=\ZZ^n
\cap \LL(\square_{\gamma})
\simeq \ZZ^{\dim \gamma+1}$ in
$\LL(\square_{\gamma}) \simeq
\RR^{\dim \gamma+1}$
we set
\begin{equation}
T_{\square_{\gamma}}:=\Spec
(\CC[M_{\gamma}]) \simeq
(\CC^*)^{\dim \gamma +1}.
\end{equation}
For $v \in M_{\gamma}$
define their lattice heights $\height
(v, \gamma) \in \ZZ$ from
$L( \gamma, Q)$ in $\LL(\square_{\gamma})$
by $\height (v,
\gamma)= \langle \alpha_{\gamma},
v \rangle$. Set $\zeta_{d_{\gamma}}
= \exp ( \frac{2 \pi \sqrt{-1}}{d_{\gamma}} )
\in \CC^*$. Then to the group
homomorphism $M_{\gamma}
\longrightarrow \CC^*$ defined by $v
\longmapsto \zeta_{d_{\gamma}}^{\height
(v, \gamma)}$ we can naturally associate an
element $\tau_{\gamma} \in
T_{\square_{\gamma}}$. We define a Laurent
polynomial $g_{\gamma}(x)=\sum_{v \in
M_{\gamma}}c_v x^v$ on $T_{\square_{\gamma}}$ by
\begin{equation}
c_v=\begin{cases}
a_v & (v \in \gamma (P)),\\
 & \\
-b_v & (v \in \gamma (Q)),\\
 & \\
\ 0 & (\text{otherwise}),
\end{cases}
\end{equation}
where $P(x)=\sum_{v \in \ZZ^n_+} a_v x^v$
and $Q(x)=\sum_{v \in \ZZ^n_+} b_v x^v$.
Then the Newton polytope $NP(g_{\gamma})$
of $g_{\gamma}$ is $\square_{\gamma}$,
$\supp g_{\gamma}
\subset \gamma (P) \sqcup \gamma (Q)$ and the
hypersurface $Z_{\square_{\gamma}}^*=
\{ x \in T_{\square_{\gamma}}\ |\
g_{\gamma}(x)=0\}$ is non-degenerate
by our assumption.
Since $Z_{\square_{\gamma}}^*
\subset T_{\square_{\gamma}}$
is invariant by
the multiplication $l_{\tau_{\gamma}}
\colon  T_{\square_{\gamma}} \simto
T_{\square_{\gamma}}$ by $\tau_{\gamma}$,
$Z_{\square_{\gamma}}^*$ admits an
action of $\mu_{d_{\gamma}}$. We thus
obtain an element
$[Z_{\square_{\gamma}}^*]$ of
$\M_{\CC}^{\hat{\mu}}$.
For a compact face $\gamma
\prec \Gamma_{+}(f)$ let $s_{\gamma} >0$
be the dimension of the minimal coordinate subspace
of $\RR^n$ containing $\gamma$ and set $m_{\gamma}=
s_{\gamma}-\dim \gamma -1 \geq 0$.
Finally, for $\lambda \in \CC$ and
an element $H \in \KK_0(\HSm)$ denote
by $H_{\lambda} \in \KK_0(\HSm)$
the eigenvalue $\lambda$-part of
$H$. Then by applying the proof of
\cite[Theorem 4.3 (i)]{M-T-4} to our
geometric situation in Theorems \ref{thm:7-6}
and \ref{thm:7-7},
we obtain the following result.

\begin{theorem}\label{thm:7-12}
Assume that $\lambda \not= 1$. Then we have the equality
\begin{equation}
[H_{f,0}^{\merc}]_{\lambda}
=\chi_h( \mathcal{S}_{f,0}^{\merc})_{\lambda}=\sum_{\gamma}
\chi_h((1-\LL)^{m_{\gamma}} \cdot
[Z_{\square_{\gamma}}^*])_{\lambda}
\end{equation}
in $\KK_0(\HSm)$,
where in the sum $\sum_{\gamma}$
the face $\gamma$ of $\Gamma_{+}(f)$
ranges through the compact ones.
\end{theorem}

\begin{proof}
For a compact face
$\gamma \prec \Gamma_{+}(f)$ of $\Gamma_{+}(f)$ set
$T_{\gamma}:=\Spec
(\CC[ \ZZ^n \cap L( \gamma, Q)]) \simeq
(\CC^*)^{\dim \gamma}$. Then we can
naturally define a Laurent polynomial
$P_{\gamma}(x)$ (resp. $Q_{\gamma}(x)$)
on it whose Newton polytope is $\gamma (P)$
(resp. $\gamma (Q)$) and the non-degenerate
hypersurface $Z_{\square_{\gamma}}^*
\subset T_{\square_{\gamma}} \simeq
(\CC^*)^{\dim \gamma +1}$ is isomorphic to
\begin{equation}
\{ (x,t) \in T_{\gamma} \times \CC^* \ | \
P_{\gamma}(x) t^{d_{\gamma}} -
Q_{\gamma}(x) = 0 \}.
\end{equation}
On the other hand, as in the proof of
\cite[Theorem 4.3]{M-T-4},
we can show that the contribution
to $\chi_h( \mathcal{S}_{f,0}^{\merc})_{\lambda} =
\chi_h( \R_{f,0}^{\merc})_{\lambda}$
for $\lambda \not= 1$ from
the compact face $\gamma$ is equal to
\begin{equation}
\chi_h((1-\LL)^{m_{\gamma}} \cdot
[Z_{\square_{\gamma}}^{\circ}])_{\lambda},
\end{equation}
where we set
\begin{equation}
Z_{\square_{\gamma}}^{\circ}
=
\Big\{ (x,t) \in T_{\gamma} \times \CC^* \ | \
P_{\gamma}(x) \cdot Q_{\gamma}(x) \not= 0,
t^{-d_{\gamma}} =
\frac{P_{\gamma}(x)}{Q_{\gamma}(x)} \Big\}.
\end{equation}
Let us set
\begin{equation}
Z_{\gamma}
=
\{ (x,t) \in T_{\gamma} \times \CC^* \ | \
P_{\gamma}(x) = Q_{\gamma}(x) = 0 \}
\subset Z_{\square_{\gamma}}^*.
\end{equation}
Then we have an equality
\begin{equation}
[Z_{\square_{\gamma}}^{\circ}] =
[Z_{\square_{\gamma}}^*] - [Z_{\gamma}].
\end{equation}
in $\M_{\CC}^{\hat{\mu}}$. Since the
restriction of $l_{\tau_{\gamma}}$ to
$Z_{\gamma}$ is homotopic to the identity,
for $\lambda \not= 1$ we obtain
\begin{equation}
\chi_h((1-\LL)^{m_{\gamma}} \cdot
[Z_{\square_{\gamma}}^{\circ}])_{\lambda}
=
\chi_h((1-\LL)^{m_{\gamma}} \cdot
[Z_{\square_{\gamma}}^{*}])_{\lambda}.
\end{equation}
This completes the proof.
\end{proof}

Now by Theorems \ref{weifilmonofil} and
\ref{thm:7-12} and Remark \ref{NJCF} 
(see also the proof of \cite[Theorem 4.3 (ii)]{M-T-4}),
we obtain the following theorem. 

\begin{theorem}\label{MAIN}
Assume that $\lambda \not= 1$ and $k \geq 1$.
Then the number of the Jordan blocks for the
eigenvalue $\lambda$ with sizes $\geq k$ in
$\Phi_{n-1,0} \colon
H^{n-1}(F_0 ;\CC) \simto
H^{n-1}(F_0 ;\CC)$ is equal to
\begin{equation}
(-1)^{n-1}\sum_{p+q=n-2+k, n-1+k}
\left\{ \sum_{\gamma}
e^{p,q} ( \chi_h ((1-\LL)^{m_{\gamma}} \cdot
[Z_{\square_{\gamma}}^*] ))_{\lambda} \right\},
\end{equation}
where in the sum $\sum_{\gamma}$
the face $\gamma$ of $\Gamma_{+}(f)$
ranges through the comapct ones.
\end{theorem}

\section{Combinatorial 
descriptions of Jordan normal forms 
and reduced Hodge spectra}\label{sec:7}

In this section, for the meromorphic 
function $f$ we give 
combinatorial descriptions of the Jordan 
normal forms of its Milnor monodromy $\Phi_{n-1, 0}$ for 
the eigenvalues $\lambda \not= 1$ 
and its reduced Hodge spectrum as in 
Matsui-Takeuchi \cite{M-T-4}, 
Stapledon \cite{Stapledon} and 
Saito \cite{Saito}. 

\subsection{Equivariant Ehrhart theory of 
Katz-Stapledon}\label{sbbsec:7}

First we recall some 
polynomials in the Equivariant Ehrhart theory 
of Katz-Stapledon \cite{Ka-St-1} and Stapledon~\cite{Stapledon}. 
Throughout this paper, we regard the empty set 
$\emptyset$ as a $(-1)$-dimensional polytope,
and as a face of any polytope.
Let $P$ be a polytope.
If a subset $F\subset P$ is a face of $P$, we write $F\prec P$.
For a pair of faces $F\prec F' \prec P$ of $P$,
we denote by $[F,F']$ the face poset $\{F''\prec 
P\mid F\prec F''\prec F'\}$,
and by $[F,F']^{*}$ a poset which is equal to 
$[F,F']$ as a set with the reversed order.

\begin{definition}
Let $B$ be a poset $[F,F']$ or $[F,F']^{*}$.
We define a polynomial $g(B,t)$ of degree 
$\leq(\dim F' -\dim F)/2$ as follows.
If $F = F'$, we set $g(B;t)=1$.
If $F \neq F'$ and $B=[F,F']$ (resp. $B=[F,F']^{*}$), 
we define $g(B;t)$ inductively by
\begin{align}
t^{\dim{F'}-\dim{F}}g(B;t^{-1})=\sum_{F''\in[F,F']}
(t-1)^{\dim{F'}-\dim{F''}}g([F,F''];t).
\\ ({\rm resp.}~t^{\dim{F'}-\dim{F}}g(B;t^{-1})=
\sum_{F''\in[F,F']^{*}}(t-1)^{\dim{F''}-\dim{F}}g([F'',F']^{*};t).) 
\end{align}
\end{definition}

In what follows, we assume that $P$ is a lattice polytope in $\RR^n$.
Let $S$ be a subset of $P\cap \ZZ^n$ containing the 
vertices of $P$, and $\omega \colon S\to \ZZ$ be a function.
We denote by $\UH_{\omega}$ the convex hull in $\RR^n\times \RR$ 
of the set $\{(v,s)\in \RR^n\times\RR \mid v\in S, s\geq \omega(v)\}$.
Then, the set of all the projections of the bounded faces of $\UH_{\omega}$ 
to $\RR^n$ defines a lattice polyhedral subdivision $\mathcal{S}$ of $P$.
Here a lattice polyhedral subdivision $\mathcal{S}$ 
of a polytope $P$ is a set of some polytopes in $P$
such that the intersection of any two polytopes in $\mathcal{S}$ is a 
face of both and all vertices of any polytope 
in $\mathcal{S}$ are in $\ZZ^{n}$. 
Moreover, the set of all the bounded faces of $\UH_{\omega}$ defines 
a piecewise $\QQ$-affine convex function $\nu\colon P\to \RR$.
For a cell $F\in\mathcal{S}$, we denote by 
$\sigma(F)$ the smallest face of $P$ containing $F$,
and $\LK_{\mathcal{S}}(F)$ the set of all cells 
of $\mathcal{S}$ containing $F$.
We call $\LK_{\mathcal{S}}(F)$ the link of $F$ in $\mathcal{S}$.
Note that $\sigma(\emptyset)=\emptyset$ and 
$\LK_{\mathcal{S}}(\emptyset)=\mathcal{S}$.

\begin{definition}
For a cell $F\in \mathcal{S}$, the $h$-polynomial 
$h(\LK_{\mathcal{S}}(F);t)$ 
of the link $\LK_{\mathcal{S}}(F)$ of $F$ is defined by
\begin{equation}
t^{\dim{P}-\dim{F}}h(\LK_{\mathcal{S}}(F);t^{-1})=
\sum_{F'\in \LK_{\mathcal{S}}(F)}g([F,F'];t)(t-1)^{\dim{P}-\dim{F'}}.
\end{equation}
The local $h$-polynomial $l_{P}(\mathcal{S},F;t)$ 
of $F$ in $\mathcal{S}$ is defined by
\begin{equation}
l_{P}(\mathcal{S},F;t)=\sum_{\sigma(F)
\prec Q\prec P}(-1)^{\dim{P}-\dim Q}h(
\LK_{{\mathcal{S}}|_{Q}}(F);t) \cdot g([Q,P]^{*};t).
\end{equation}
\end{definition}

For $\lambda\in \CC$ and 
$v\in mP\cap{\ZZ^n}$ ($m\in \ZZ_+:=\ZZ_{\geq 0}$) we set 
\begin{align}
w_{\lambda}(v)=
\left\{
\begin{array}{ll}
1& \Bigl( \exp\ \bigl( 2\pi\sqrt{-1}\cdot m\nu(\frac{v}{m})\bigr)=
\lambda \Bigr) \\\\
0&( \text{otherwise} ). \\
\end{array}
\right.
\end{align}
We define the $\lambda$-weighted Ehrhart polynomial 
$\phi_{\lambda}(P,\nu;m)\in\ZZ[m]$ 
of $P$ with respect to $\nu:P\to \RR$ by
\begin{equation}
\phi_{\lambda}(P,\nu;m):=\sum_{v\in {mP}\cap\ZZ^{n}}w_{\lambda}(v).
\end{equation}
Then $\phi_{\lambda}(P,\nu;m)$ is a polynomial in $m$ 
with coefficients $\ZZ$ 
whose degree is $\leq\dim P$ (see \cite{Stapledon}).  

\begin{definition}(\cite{Stapledon})
\begin{enumerate}
\item[\rm{(i)}] We define the $\lambda$-weighted $h^{*}$-polynomial 
$h^{*}_{\lambda}(P,\nu;u)
\in \ZZ[u]$ by 
\begin{equation}
\sum_{m\geq 0} \phi_{\lambda}(P,\nu;m)u^{m}=
\frac{h^{*}_{\lambda}(P,\nu;u)}{(1-u)^{\dim{P}+1}}.
\end{equation}
If $P$ is the empty polytope, we set $h^{*}_{1}
(P,\nu;u)=1$ and $h^{*}_{\lambda}(P,\nu;u)=0~(\lambda\neq 1)$.
\item[\rm{(ii)}] We define the $\lambda$-local weighted 
$h^{*}$-polynomial $l^{*}_{\lambda}(P,\nu;u)\in\ZZ[u]$ by
\begin{equation}
l^{*}_{\lambda}(P,\nu;u)=\sum_{Q\prec P}
(-1)^{\dim{P}-\dim{Q}}h^{*}_{\lambda}(Q,\nu|_{Q};u)
 \cdot g([Q,P]^{*};u).
\end{equation}
If $P$ is the empty polytope, we set $l^{*}_{1}(P,\nu;u)=1$ 
and $l^{*}_{\lambda}(P,\nu;u)=0~(\lambda\neq 1)$.
\end{enumerate}
\end{definition} 

\begin{definition}(\cite{Stapledon})\label{def:poly}
\begin{enumerate}
\item[\rm{(i)}] We define the $\lambda$-weighted limit mixed 
$h^{*}$-polynomial $h^{*}_{\lambda}(P,\nu;u,v)\in\ZZ[u,v]$ by
\begin{equation}
h^{*}_{\lambda}(P,\nu;u,v):=\sum_{F\in\MCS}v^{\dim{F}+1}
l^{*}_{\lambda}(F,\nu|_{F};uv^{-1}) \cdot h(\LK_{\MCS}(F);uv).
\end{equation}
\item[\rm{(ii)}] We define the $\lambda$-local weighted limit mixed 
$h^{*}$-polynomial $l^{*}_{\lambda}(P,\nu;u,v)\in\ZZ[u,v]$ by
\begin{equation}
l^{*}_{\lambda}(P,\nu;u,v):=\sum_{F\in\MCS}v^{\dim{F}+1}
l^{*}_{\lambda}(F,\nu|_{F};uv^{-1}) \cdot l_{P}(\MCS,F;uv).
\end{equation}
\end{enumerate}
\end{definition}

\subsection{Jordan normal forms and reduced Hodge spectra}\label{sbbsec:8}

Assume that the meromorphic function $f(x)= \frac{P(x)}{Q(x)}$ is 
non-degenerate at the origin $0 \in X= \CC^n$ 
and $P(x), Q(x)$ are convenient. For $f$ to be 
polynomial-like, we assume moreover that 
$\Gamma_+(P)$ is properly contained in 
$\Gamma_+(Q)$: $\Gamma_+(P) \subset 
\subset \Gamma_+(Q)$ (see Definition \ref{def-29}). 
We call the union of the compact faces of 
$\Gamma_+(P)$ (resp. $\Gamma_+(Q)$) the Newton boundary of 
$P$ (resp. $Q$) and denote it by $\Gamma_P$ 
(resp. $\Gamma_Q$). 
Denote by $K$ the convex hull 
of the closure of 
$\Gamma_+(Q) \setminus \Gamma_+(P)$ in $\RR^n$ and
define a piecewise $\QQ$-affine function 
$\nu$ on $K$ which takes the value $1$ (resp. $0$)
on $\Gamma_Q \subset \RR^n$ 
(resp. on the convex hull of 
$\Gamma_P \subset \RR^n$) such that 
for any compact face $\gamma$ of 
$\Gamma_+(f)$ 
the restriction of $\nu$ to 
$\square_{\gamma}$ is an affine function. 
For $\lambda \in \CC$ we define 
the equivariant Hodge-Deligne polynomial for the eigenvalue $\lambda$ 
(of the mixed Hodge structures of the cohomology groups of the Milnor 
fiber $F_{0}$) $E_{\lambda}(F_{0};u,v) \in \ZZ [u,v]$ by 
\begin{equation}
E_{\lambda}(F_{0};u,v)= 
\sum_{p,q\in\ZZ}
\sum_{j\in\ZZ}(-1)^{j}h^{p,q}_{\lambda}(H^{j}(F_{0};\CC))
u^pv^q\in\ZZ[u,v],
\end{equation}
where $h^{p,q}_{\lambda}(H^{j}(F_{0};\CC))$ is the dimension of 
$\GR^{p}_{F}\GR^{W}_{p+q}H^{j}(F_{0};\CC)_{\lambda}$. 
Then for $\lambda \not= 1$, as in \cite{M-T-4}, 
\cite{Stapledon} and \cite{Saito}, 
by Theorem \ref{thm:7-12} we can 
calculate the $\lambda$-part of  
the Hodge realization 
of the motivic Milnor fiber 
$\mathcal{S}^{\merc}_{f,0}$ of $f$ and 
obtain the following formula for 
$E_{\lambda}(F_{0};u,v)$. 

\begin{theorem}\label{the-3214} 
In the situation as above, for any $\lambda \not= 1$ 
we have 
\begin{equation}
uv E_{\lambda}(F_{0};u,v)
=(-1)^{n-1}l^{*}_{\lambda}(K,\nu;u,v). 
\end{equation}
\end{theorem}

Let $\mathcal{S}_{\nu}$ be the polyhedral 
subdivision of the polytope $K$ defined by $\nu$. 
By the definition of the $h^*$-polynomial, 
for $\lambda \not= 1$ we have
\begin{equation}
l^{*}_{\lambda}(K,\nu;u,u)=
\sum_{\gamma \prec \Gamma_+(f) :{\rm compact}}
u^{\dim \square_{\gamma} +1}
l^{*}_{\lambda}( \square_{\gamma} ,\nu;1)
\cdot l_{K}( \mathcal{S}_{\nu}, \square_{\gamma} ;u^2),
\end{equation}
where in the sum $\Sigma$ the face $\gamma$ 
ranges through the compact ones of $\Gamma_+(f)$.
The polynomial $l_{K}( \mathcal{S}_{\nu}, \square_{\gamma} ;t)$ 
is symmetric and unimodal centered at 
$(n-\dim \gamma -1)/2$, i.e. if $a_{i}\in \ZZ$ 
is the coefficient of $t^i$ in $l_{K}
( \mathcal{S}_{\nu}, \square_{\gamma} ;t)$ we have 
$a_{i}=a_{n-\dim \gamma -1-i}$ and $a_{i}\leq a_{j}$ 
for $0\leq i\leq j\leq (n-\dim \gamma -1)/2$.
Therefore, it can be expressed in the form
\begin{equation}
l_{K}( \mathcal{S}_{\nu}, \square_{\gamma} ;t)=
\sum_{i=0}^{\lfloor (n-1-\dim \gamma )/2\rfloor}
\tl{l}_{\gamma ,i}(t^i+t^{i+1}+\dots+t^{n-1-\dim \gamma -i}),
\end{equation}
for some non-negative integers $\tl{l}_{{\gamma},i}\in\ZZ_{\geq 0}$.
We set
\begin{equation}
\tl{l}_{K}( \mathcal{S}_{\nu}, \square_{\gamma} ,t):=
\sum_{i=0}^{\lfloor (n-1-\dim \gamma )/2\rfloor}
\tl{l}_{\gamma,i}t^i.
\end{equation}
For $k\in\ZZ_{> 0}$ and $\lambda \in\CC$ 
we denote by $J_{k,\lambda}$ the number of 
the Jordan blocks in $\Phi_{n-1, 0}$ with size $k$ 
for the eigenvalue $\lambda$.
Then we obtain the following formula for them.

\begin{corollary}\label{jordan}
In the situation as above, 
for any $\lambda \not= 1$ we have
\begin{equation}
\sum_{0\leq k\leq n-1}J_{n-k,\lambda}u^{k+2}
=
\sum_{\gamma \prec \Gamma_+(f) :{\rm compact}}
u^{\dim \square_{\gamma} +1}l^{*}_{\lambda}
(\square_{\gamma},\nu;1)\cdot\tl{l}_{K}( \mathcal{S}_{\nu},
\square_{\gamma} ;u^2),
\end{equation}
where in the sum $\Sigma$ of the right 
hand side the face $\gamma$ ranges through 
the compact ones of $\Gamma_+(f)$.
\end{corollary}

Let $q_1,\ldots,q_l$ (resp. $\gamma_1,\ldots, 
\gamma_{l^{\prime}}$) be the 
$0$-dimensional (resp. $1$-dimensional) 
faces of $\Gamma_{+}(f)$ such 
that $q_i\in \Int ( \RR_+^n )$ (resp. 
the relative interior 
$\relint(\gamma_i)$ of $\gamma_i$ is 
contained in $\Int( \RR_+^n )$). 
For each $q_i$ (resp. $\gamma_i$), 
we set $d_i:=d_{q_i} >0$ (resp. $e_i:=
d_{\gamma_i}>0$). Moreover, for $\lambda 
 \not= 1$ and $1 \leq 
i \leq l^{\prime}$ such that $\lambda^{e_i}=1$ we set
\begin{align}
n(\lambda)_i
 := & \sharp\{ v\in \ZZ^n \cap \relint ( \square_{\gamma_i} ) 
\ |\ \height (v)_i =k\} 
\nonumber 
\\
 & +\sharp \{ v\in \ZZ^n \cap \relint ( \square_{\gamma_i} ) 
\ |\ \height (v)_i=e_i-k\},
\end{align}
where $k$ is the minimal positive 
integer satisfying 
$\lambda=\zeta_{e_i}^{k}$ 
and for $v\in \ZZ^n \cap \relint 
( \square_{\gamma_i} )$ we 
denote by $\height (v)_i>0$ the 
lattice height of $v$ from the hyperplane 
$H( \gamma_i, Q) \subset 
{\rm Aff}( \square_{\gamma_i} )$. 
Then we have the following generalization of 
\cite[Theorem 4.4]{M-T-4} to polynomial-like 
rational functions. 

\begin{theorem}\label{NJBS} 
In the situation as above, for any 
$\lambda \not= 1$ we have 
\begin{enumerate}
\item[\rm{(i)}]  The number of the Jordan blocks 
for the eigenvalue $\lambda$ with the 
maximal possible size $n$ in 
$\Phi_{n-1,0}: H^{n-1}(F_0; \CC ) \simto H^{n-1}(F_0; \CC )$ 
is equal to $\sharp \{q_i \ |\ \lambda^{d_i}=1\}$.
\item[\rm{(ii)}]  The number of the Jordan blocks for 
the eigenvalue $\lambda$ with the second maximal possible size 
$n-1$ in $\Phi_{n-1,0}$ is 
equal to $\sum_{i \colon \lambda^{e_i}=1} 
n(\lambda)_i$.
\end{enumerate}
\end{theorem}

For a compact face $\gamma \prec \Gamma_+(f)$ 
we define a Puiseux 
series $h_{\gamma}(t)$ with coefficients in $\ZZ$ by 
\begin{equation}
h_{\gamma}(t)
= \sum_{\beta \in (\QQ \setminus \ZZ) \cap (0,+ \infty )}
\Bigl\{ 
\phi_{\exp(2\pi\sqrt{-1}\beta)}( \square_{\gamma}, \nu ; 
\lfloor{\beta}\rfloor +1) - 
\phi_{\exp(2\pi\sqrt{-1}\beta)}( \square_{\gamma}, \nu ; 
\lfloor{\beta}\rfloor )
\Bigr\} 
t^{\beta}. 
\end{equation} 

Then as in \cite[Theorem 5.16]{M-T-3} (or 
\cite[Corollary 5.3]{Saito}), by Theorem \ref{thm:7-12} 
(or Theorem \ref{the-3214}) we obtain the following 
result. 

\begin{corollary}\label{spfor}
In the situation as above, we have
\begin{equation}
\tl{\rm sp}_{f,0}(t) = 
\sum_{\gamma \prec \Gamma_+(f) :{\rm compact}}
(-1)^{n-1-\dim \gamma} (1-t)^{s_{\gamma}}
\cdot h_{\gamma}(t), 
\end{equation}
where in the sum $\Sigma$ of the right 
hand side the face $\gamma$ ranges through 
the compact ones of $\Gamma_+(f)$.
\end{corollary}

%Note that this result is an analogue of 
%the classical one for polynomial 
%functions $f$ proved in \cite{V-K} and \cite{SM}. 

\section{Monodromies at infinity of rational functions}\label{sec:9}

In this section, we consider monodromies at infinity of 
rational functions. Let $X$ be a smooth and connected algebraic 
variety over $\CC$ and $P(x), Q(x)$ 
regular functions on it. Assume that $Q(x)$ 
is not identically zero on 
$X$. We define a rational function $f(x)$ on 
$X$ by 
\begin{equation}
f(x)= \frac{P(x)}{Q(x)}  \qquad (x \in X). 
\end{equation}
Then there exists a fnite subset $B \subset \CC$ 
such that $f: X \setminus Q^{-1}(0) 
\rightarrow \CC$ induces a locally trivial fibration 
\begin{equation}
f^{-1}(\CC \setminus B) \rightarrow  \CC \setminus B. 
\end{equation}
The smallest finite subset $B \subset \CC$ 
satisfying this property is called the 
bifurcation set of $f$. For the study of 
such subsets for regular and rational functions, 
see e.g. \cite{B-P}, 
\cite{C-D-T-T}, \cite{G-L-M-new-new}, 
\cite{Thang}, \cite{Nguyen}, \cite{N-S-T},  
\cite{Tak}, \cite{Zaharia} etc. 
For large enough $R \gg 0$ let
\begin{equation}
\Phi_{j}^{\infty}: H^{j}(f^{-1}(R); \CC )
\simto H^{j}(f^{-1}(R); \CC ) \qquad (j \in \ZZ ) 
\end{equation}
be the monodromy operators associated 
to the preceding fibration. 
Then we define 
the monodromy zeta function $\zeta_{f}^{\infty}(t) \in \CC (t)$ 
at infinity of $f$ by 
\begin{equation}
\zeta_{f}^{\infty}(t) = \prod_{j \in \ZZ} 
\Bigl\{ {\rm det}( \id - t \Phi_{j}^{\infty}) 
\Bigr\}^{(-1)^j}  \in \CC (t). 
\end{equation}
We have the following 
global analogue for $\zeta_{f}^{\infty}(t)$ of 
A'Campo's formula in \cite{Acampo}, 
which could be deduced also from the proof of 
\cite[Theorem 1]{G-L-M-new} and the 
results in \cite[Section 6.1]{Dimca}. Here we give a 
short proof 
to it for the reader's convenience. 

\begin{theorem}\label{the-51} 
Assume that the hypersurfaces $P^{-1}(0)$ and 
$Q^{-1}(0)$ are smooth and intersect 
transversally in $X$. 
Let $\overline{X} \supset X$ be a smooth compactification of 
$X$ such that for the complement 
$D:= \overline{X} \setminus X$ the 
union $D \cup P^{-1}(0) \cup Q^{-1}(0) \subset \overline{X}$ 
is a strict normal crossing divisor in $\overline{X}$. 
Let $D= \cup_{i=1}^r D_i$ be the irreducible 
decomposition of $D$. For $1 \leq i \leq r$ set 
\begin{equation}
D_i^{\circ}= D_i \setminus (\cup_{j \not= i} D_j 
\cup P^{-1}(0) \cup Q^{-1}(0))
\end{equation}
and 
\begin{equation}
l_i= {\rm ord}_{D_i} (Q) - 
{\rm ord}_{D_i} (P) \in \ZZ. 
\end{equation}
Then we have 
\begin{equation}
\zeta_{f}^{\infty}(t) = 
(1-t)^{\chi (Q^{-1}(0) \setminus P^{-1}(0))} \cdot 
\Bigl\{ \prod_{i: l_i >0} 
(1- t^{l_i})^{\chi ( D_i^{\circ} )} \Bigr\}. 
\end{equation}
\end{theorem}

\begin{proof}
Let $h:\PP^1 \rightarrow  \CC$ be a local coordinate at 
$\infty \in \PP$ such that $h(\infty)=0$. 
By $f: X \setminus Q^{-1}(0) 
\rightarrow \CC$ and the inclusion map 
$j: \CC \hookrightarrow \PP^1$ we set 
\begin{equation}
\F : = j_{!}Rf_{!}\CC_{X \setminus Q^{-1}(0)} 
\in \Dbc ( \PP ).
\end{equation}
Then we have 
$\zeta^{\infty}_f(t)= 
\zeta( \F )_{h, \infty}(t)$. Indeed, $\zeta^{\infty}_f(t)$ 
is equal to the zeta function associated to the 
monodromy automorphisms
\begin{equation}
H_j(f^{-1}(R); \CC ) \simto H_j(f^{-1}(R); \CC ) \qquad (j \in \ZZ ).
\end{equation}
Moreover we have the Poincar\'e duality isomorphisms
\begin{equation}
H_j(f^{-1}(R); \CC ) \simeq H_c^{2n-j-2}(f^{-1}(R); \CC ) 
\qquad (j \in \ZZ ).
\end{equation}
Hence the monodromy zeta function at infinity $\zeta^{\infty}_f(t)$ is 
equal to the one associated to the monodromy automorphisms 
\begin{equation}
H^j_c(f^{-1}(R); \CC ) \simto H^j_c(f^{-1}(R); \CC ) 
\qquad (j \in \ZZ ).
\end{equation}
As in the proof of \cite[Theorem 3.6]{M-T-2} we 
construct towers of blow-ups of $\overline{X}$ over 
the normal crossing divisor 
$D \cup P^{-1}(0) \cup Q^{-1}(0)$ to eliminate the 
points of indeterminacy of $f$. Then we 
obtain a proper morphism $\pi : Y \longrightarrow \overline{X}$ 
of complex manifolds which induces an isomorphism 
\begin{equation}
Y \setminus \pi^{-1}(D \cup P^{-1}(0) \cup Q^{-1}(0)) 
\simto \overline{X} \setminus (D \cup P^{-1}(0) \cup Q^{-1}(0))
\end{equation}
and the assertion follows from the proof of 
\cite[Theorem 3.6]{M-T-2}. 
\end{proof}

\begin{definition}\label{def-2-7}
We say that the rational function 
$f(x)= \frac{P(x)}{Q(x)}$ is polynomial-like 
at infinity if it satisfies the assumptions of 
Theorem \ref{the-51} and there exists a 
smooth compactification $\overline{X} \supset X$ of 
$X$ (satisfying the condition in Theorem \ref{the-51}) 
such that for any irreducible component $D_i$ of 
$D= \overline{X} \setminus X$ we have the condition 
\begin{equation}
{\rm ord}_{D_i} (P) < 
{\rm ord}_{D_i} (Q)
\end{equation}
i.e. $f$ has a pole of order ${\rm ord}_{D_i} (Q) 
- {\rm ord}_{D_i} (P)>0$ along $D_i$. 
\end{definition}

For $j \in \ZZ$ and $\lambda \in \CC$ 
and $R \gg 0$ 
we denote by 
\begin{equation}
H^j(f^{-1}(R); \CC )_{\lambda} \subset H^j(f^{-1}(R); \CC )
\end{equation}
the generalized eigenspace of $\Phi_{j}^{\infty}: 
H^{j}(f^{-1}(R); \CC )\simto H^{j}(f^{-1}(R); \CC )$ 
for the eigenvalue $\lambda$. 
Then as in Takeuchi-Tib{\u a}r \cite{T-T} 
we obtain the following result. 

\begin{theorem}\label{the-3211} 
Assume that $X$ is affine and 
the rational function 
$f(x)= \frac{P(x)}{Q(x)}$ is polynomial-like 
at infinity. 
Then for any $\lambda \not=1$ we have 
the concentration 
\begin{equation}
H^j(f^{-1}(R); \CC )_{\lambda} \simeq 0 
\qquad (j \not= \dim X-1). 
\end{equation}
Moreover the weight filtration on 
$H^{\dim X-1}(f^{-1}(R); \CC )_{\lambda}$ 
coincides with the monodromy filtration of 
$\Phi_{n-1}^{\infty}$. 
\end{theorem}

Combining the formula for 
$\zeta_{f}^{\infty}(t) \in \CC (t)$ in Theorem \ref{the-51} 
with Theorem \ref{the-3211} above, we obtain a 
formula for the multiplicities of the eigenvalues 
$\lambda \not= 1$ in $\Phi_{\dim X-1}^{\infty}$. 

From now on, we consider the special case where $X= \CC^n$ 
and $P(x), Q(x)$ are convenient polynomials. 
Let $\Gamma_{\infty}(P) 
\subset \RR^n$ (resp. $\Gamma_{\infty}(Q) \subset \RR^n$) 
be the convex hull of $\{ 0 \} \cup NP(P)$ 
(resp. $\{ 0 \} \cup NP(Q)$) in $\RR^n$ and 
\begin{equation}
\Gamma_{\infty}(f)= \Gamma_{\infty}(P) + \Gamma_{\infty}(Q) 
\end{equation}
their Minkowski sum. Since $P(x), Q(x)$ are 
convenient, they are $n$-dimensional polytopes in $\RR^n$. 
As in the case of $\Gamma_{+}(f)$,  
for each face 
$\gamma \prec \Gamma_{\infty}(f)$ 
we have the corresponding faces 
\begin{equation}
\gamma(P) \prec \Gamma_{\infty}(P), \qquad 
\gamma(Q) \prec \Gamma_{\infty}(Q)
\end{equation}
such that 
\begin{equation}
\gamma = \gamma(P) + \gamma(Q). 
\end{equation}

\begin{definition}\label{def-3101}
We say that the rational 
function $f(x)= \frac{P(x)}{Q(x)}$ is 
non-degenerate at infinity if for any 
face $\gamma$ of $\Gamma_{\infty}(f)$ 
such that $0 \notin \gamma$ the complex 
hypersurfaces $\{x \in T=(\CC^*)^n\ |\ 
P^{\gamma(P)}(x)=0\}$ and $\{x \in T=(\CC^*)^n\ |\ 
Q^{\gamma(Q)}(x)=0\}$ are smooth and reduced 
and intersect transversally in $T=(\CC^*)^n$. 
\end{definition}

For a subset 
$S \subset \{ 1,2, \ldots, n \}$ we set 
\begin{equation}
\Gamma_{\infty}(f)^S= \Gamma_{\infty}(f) \cap \RR^S. 
\end{equation}
Similarly, we define 
$\Gamma_{\infty}(P)^S, \Gamma_{\infty}(Q)^S \subset \RR^S$ 
so that we have 
\begin{equation}
\Gamma_{\infty}(f)^S= \Gamma_{\infty}(P)^S + \Gamma_{\infty}(Q)^S. 
\end{equation}
Let $\gamma_1^S, \gamma_2^S, \ldots, \gamma_{n(S)}^S$ 
be the facets of $\Gamma_{\infty}(f)^S$ 
such that $0 \notin \gamma_i^S$ and for 
each $\gamma_i^S$ ($1 \leq i \leq n(S)$) consider 
the corresponding faces 
\begin{equation}
\gamma_i^S(P) \prec \Gamma_{\infty}(P)^S, \qquad 
\gamma_i^S(Q) \prec \Gamma_{\infty}(Q)^S
\end{equation}
such that 
\begin{equation}
\gamma_i^S= \gamma_i^S(P) + \gamma_i^S(Q). 
\end{equation}
By using the primitive outer conormal vector 
$\alpha_i^S \in \ZZ_+^S \setminus \{ 0 \}$ 
of the facet $\gamma_i^S \prec \Gamma_{\infty}(f)^S$ 
we define the lattice distance 
$d_i^S(P)>0$ (resp. $d_i^S(Q)>0$) of 
$\gamma_i^S(P)$ (resp. $\gamma_i^S(Q)$) from 
the origin $0 \in \RR^S$ and set 
\begin{equation}
d_i^S=d_i^S(P)-d_i^S(Q) \in \ZZ. 
\end{equation}
Finally we define $v_i^S >0$ as in Section \ref{sec:s3}. 
Then we obtain the following result. 

\begin{theorem}\label{the-2-1} 
Assume that the rational 
function $f(x)= \frac{P(x)}{Q(x)}$ is 
non-degenerate at infinity and 
the hypersurfaces $P^{-1}(0)$ and 
$Q^{-1}(0)$ are smooth and intersect 
transversally in $X= \CC^n$.  
Then we have 
\begin{equation}
\zeta_{f}^{\infty}(t) = 
(1-t)^{\chi (Q^{-1}(0) \setminus P^{-1}(0))} \cdot 
\prod_{S \not= \emptyset} \Bigl\{ 
\prod_{i: d_i^S >0} 
( 1- t^{d_i^S})^{(-1)^{|S|-1} v_i^S} \Bigr\}. 
\end{equation}
\end{theorem}

\begin{proof}
Let $\Sigma_f$ be the dual fan of 
$\Gamma_{\infty}(f)$ in $\RR^n$. Since 
$P(x), Q(x)$ are convenient, any face of 
$\RR^n_+$ is a cone in it. Then we can 
construct a smooth subdivision $\Sigma$ 
of $\Sigma_f$ without subdividing such 
cones. In other words, the fan 
$\Sigma_0$ in $\RR^n_+$ 
formed by all the faces of $\RR^n_+$ is 
a subfan of $\Sigma$. 
Denote by $X_{\Sigma}$ the smooth 
toric variety associated to it and 
containing $X= \CC^n$. Then we 
obtain the assertion just by applying 
Theorem \ref{the-51} to the smooth 
compactification $X_{\Sigma} \supset X$ 
of $X= \CC^n$. 
\end{proof}

If moreover $\Gamma_{\infty}(Q) \subset \subset 
\Gamma_{\infty}(P)$, then the rational function 
$f(x)= \frac{P(x)}{Q(x)}$ is polynomial-like 
at infinity in the sense of Definition 
\ref{def-3101} and we obtain also  
combinatorial descriptions 
of the Jordan normal forms for 
the eigenvalues 
$\lambda \not= 1$ in $\Phi_{n-1}^{\infty}$ 
and its reduced Hodge spectra at infinity. 
We leave their precise formulations to the 
readers.

\appendix
\section{Proof of Lemma~\ref{weifillem} by 
Takahiro Saito}\label{app} 

In this appendix, we prove Lemma~\ref{weifillem} 
in the main paper. For $k \in \ZZ$ we set 
\begin{equation}
L_{k}\psi_{g,\lambda}^{H}(M):=
\psi_{g,\lambda}^{H}(W_{k+1}M), \qquad 
L_{k}\psi^{H}_{g,\lambda}(M'):=
\psi_{g,\lambda}^{H}(W_{k+1}M'). 
\end{equation}
Recall that the weight filtration 
$W_{\bullet}\psi^{H}_{g,\lambda}(M)$ of 
$\psi_{g,\lambda}^{H}(M)$ is the relative monodromy 
filtration with respect to the filtration 
$L_{\bullet}\psi^{H}_{g,\lambda}(M)$. Namely for any 
$k\in \ZZ$ the filtration on $\GR^{L}_{k}
\psi_{g,\lambda}^{H}(M)$ induced by the weight 
filtration $W_{\bullet}\psi^{H}_{g,\lambda}(M)$ 
is the monodromy filtration centered at $k$.
Therefore, it is enough to show that 
\begin{equation}
\GR^{L}_{k} \psi^{H}_{g,\lambda}(M)=0 
\end{equation}
for any $k\neq l-1$. 
Take a sufficiently large $k_{0}(>l-1)$ such that 
\begin{equation}
L_{k_{0}}\psi_{g,\lambda}^{H}(M)=
\psi_{g,\lambda}^{H}(M), \qquad 
L_{k_{0}}\psi^{H}_{g,\lambda}(M')=
\psi_{g,\lambda}^{H}(M'). 
\end{equation}
Since we have $\psi^{H}_{g,\lambda}(M) \simto  
\psi^{H}_{g,\lambda}(M')$, 
for such $k_{0}$ the morphism $\GR^{L}_{k_{0}}
\psi^{H}_{g,\lambda}(M) \to \GR^{L}_{k_{0}} 
\psi^{H}_{g,\lambda}(M')$ is an epimorphism. 
On the other hand, by the exactness of the functor 
$\psi^{H}_{g,\lambda}( \cdot )$, it follows from 
our assumption on the weights of $M$ that 
we have $\GR^{L}_{k_{0}} 
\psi^{H}_{g,\lambda}(M)=\psi^{H}_{g,\lambda}(
\GR^{W}_{k_{0}+1}M)=0$. 
Therefore, $\GR^{L}_{k_{0}}\psi^{H}_{g,\lambda}(M')$ 
is also zero and hence we obtain 
\begin{equation}
\psi_{g,\lambda}^{H}(M')= L_{k_{0}}\psi^{H}_{g,\lambda}(M')
=L_{k_{0}-1} \psi^{H}_{g,\lambda}(M'). 
\end{equation}
Repeating this argument, we get $\GR^{L}_{k}
\psi_{g,\lambda}^{H}(M')=0$ for any $k>l-1$.
Similarly, we can show that $\GR^{L}_{k}
\psi_{g,\lambda}^{H}(M')=0$ for any $k\neq l-1$. 
This completes the proof.

\end{document}